\documentclass[10pt]{amsart}
\usepackage{amsthm, amsmath, hyperref, cleveref, amsfonts, amssymb, enumerate, textcmds, tensor, enumitem, mathdots}
\usepackage{pst-node}
\usepackage{tikz-cd} 
\usepackage{graphicx,tikz}
\usepackage{enumerate}
\usepackage{pinlabel}
\usepackage[colorinlistoftodos]{todonotes}
\usepackage{commath}
\usepackage[Bourbaki-arrow]{dynkin-diagrams}
\usepackage{placeins}
\usepackage{mathtools}
\newcommand{\sign}[1]{\mathrm{sign}\left(#1\right)}

\let\svthefootnote\thefootnote
\newcommand\freefootnote[1]{%
  \let\thefootnote\relax%
  \footnotetext{#1}%
  \let\thefootnote\svthefootnote%
}

\newcommand{\R}{\mathbb R}

\newcommand{\p}{\mathfrak{p}}
\newcommand{\g}{\mathfrak{g}}
\newcommand{\kk}{\mathfrak{k}}

\newtheorem{thm}{Theorem}[section]
\newtheorem{lem}[thm]{Lemma}
\newtheorem{prop}[thm]{Proposition}

\newtheorem{que}[thm]{Question}

\newtheorem{claim}[thm]{Claim}

\newtheorem*{thma}{Theorem A}
\newtheorem*{thmb}{Theorem B}

\newtheorem*{thmc}{Theorem C}
\newtheorem*{thmcp}{Theorem C'}

\newtheorem*{corb}{Corollary B}

\DeclareMathOperator{\C}{\mathbb{C}}

\theoremstyle{definition}
\newtheorem{defn}[thm]{Definition}
\newtheorem{remark}[thm]{Remark}

\begin{document}
\title{On Hitchin's equations for cyclic $G$-Higgs bundles}
\author{Nathaniel Sagman}
\author{Ognjen To{\v s}i{\'c}}

\begin{abstract}
   We develop a Lie-theoretic perspective on Hitchin's equations for cyclic $G$-Higgs bundles, which we use to study analytic and geometric properties of harmonic maps. Among other things, we prove Dai-Li's conjecture on the monotonicity of the energy density in the case of Coxeter cyclic $G$-Higgs bundles, for all $G$, and Dai-Li's negative curvature conjecture for Coxeter cyclic $G$-Higgs bundles, for all $G$ except those of type $\mathrm{E}_7$ and $\mathrm{E}_8.$
\end{abstract}
\maketitle
\section{Introduction}
Let $S$ be a Riemann surface with canonical bundle $\mathcal{K}$ and let $G$ be a simple complex Lie group with Lie algebra $\g$. A $G$-Higgs bundle on $S$ is a pair $(P,\phi)$, where $P$ is a holomorphic principal $G$-bundle and $\phi$ is a holomorphic $(1,0)$-form valued in the adjoint bundle $\textrm{ad}P$ called the Higgs field. Hitchin's self-duality equations ask for a reduction of structure group of $P$ to a maximal compact subgroup $K\subset G$ that gives rise to a Chern connection $\nabla_h$ on $P$ and adjoint operator $*_h$ on $\textrm{ad}P$ satisfying $$F(\nabla_h) + [\phi,\phi^{*_h}]=0.$$
A solution is equivalent to an equivariant harmonic map from the universal cover $\tilde{S}$ to the Riemannian symmetric space $G/K$ of $G$. When $S$ is closed and of genus $g\geq 2$, for a generic class of $G$-Higgs bundles, namely the stable $G$-Higgs bundles, one can always find a unique solution. The self-duality equations and existence theory for harmonic maps together yield the non-abelian Hodge correspondence, a bridge between the moduli spaces of representations $\pi_1(S)\to G$ and $G$-Higgs bundles. 

\begin{defn}
   A $G$-Higgs bundle $(P,\phi)$ is cyclic if there exists a holomorphic gauge transformation $s$ of $P$ of order $m$ such that $s^*\phi=e^{\frac{2\pi i}{m}}\phi.$
\end{defn}
There are different notions of cyclic $G$-Higgs bundles in the literature; see \S \ref{sec: cyclic definition}.

 For $m>2,$ harmonic maps arising from cyclic $G$-Higgs bundles are weakly conformal, and hence define branched minimal immersions.  Moreover, cyclic Higgs bundles often give rise to other distinguished immersions (hyperbolic affine spherical immersions into $\R^3$, $J$-holomorphic maps into $\mathbb{H}^{2,4},$ etc.) and they play an important role in higher Teichm{\"u}ller theory. For example, Labourie's proof of his own conjecture in rank $2$ uses cyclic Higgs bundles in an essential way \cite{L2}. For studies on cyclic Higgs bundles, see \cite{Baraglia2010CyclicHB}, \cite{B}, \cite{C}, \cite{Dai2018},  \cite{GP}, \cite{GPRi}, \cite{GPG}, \cite{L2}, \cite{Katz}, among other works. 
 
In this paper, we develop a framework to study Hitchin's equations for cyclic $G$-Higgs bundles, for all $G$, in a unified way. We then use what we've found to study analytic and geometric properties of harmonic maps.

\subsection{Affine Toda equations}
 Our starting point is as follows. The most well-studied cyclic $G$-Higgs bundles are those for $G=\textrm{SL}(n,\C)$ with order $m$ equal to $n$. In this case, Hitchin's equations are equivalent to an elliptic system for $n$ sections of line bundles, often called affine Toda equations, and have been studied extensively. See, for instance, \cite{CL}, \cite{DL}, \cite{Dai2018}, \cite{LM1}. Among this class, particular focus has been given to the Higgs bundles in the Hitchin section (see \S \ref{sec: hitchin section}), for which the affine Toda equations recover well-known equations from geometry, such as the Bochner equation for harmonic maps to $\mathbb{H}^2$ (for $n=2$) or the Tzit{\'e}ica equation for hyperbolic affine spheres in $\R^3$ (for $n=3$). 

Our first result (Theorem A below) concerns groups of adjoint type; it has a version for more general groups, but it is more cleanly stated with the adjoint assumption. Associated with a Cartan subalgebra $\mathfrak{h}$ of $\g$, we have the root system $\Delta$ with Coxeter number $r$ (see Definition \ref{def: coxeter}). For $\g=\mathfrak{s}\mathfrak{l}(n,\C)$, $r=n.$ When the order $m$ is $r,$ we say that $(P,\phi)$ is Coxeter cyclic. Given a choice of simple roots $\Pi$ for $\mathfrak{h}$, adjoining the lowest root $-\delta$ gives the extended system of simple roots $\mathcal{Z}=\{-\delta\}\cup \Pi$. For $m=r,$ we find that, exactly in this case and for no other $m$, Hitchin's equations split into a system of equations for a set of Hermitian metrics on line bundles, indexed by $\mathcal{Z}$. More precisely, for $m=r$, we first prove the following, and then we prove a result on the structure of Hitchin's equations for Coxeter cyclic $G$-Higgs bundles, Theorem A below.
\begin{prop}\label{prop: structure}
        Let $G$ be a simple complex Lie group of adjoint type. The data of a Coxeter cyclic $G$-Higgs bundle $(P,\phi)$ over a Riemann surface $S$ is equivalent to a collection of line bundles $\{L_\alpha:\alpha\in\mathcal{Z}\}$, holomorphic sections $\phi_\alpha\in H^0(S,L_\alpha\otimes \mathcal{K})$, and an isomorphism $\Theta$ from $\bigotimes_{\alpha\in\mathcal{Z}} L_\alpha^{\otimes n_\alpha}$ to the trivial line bundle $\mathcal{O}$. 
\end{prop}
The decomposition of $\phi$ can be cast as a special case of \cite[Theorem 4.2.2]{C}, where the extra hypothesis gives more structure. About (poly)stability of $(P,\phi)$ in terms of the $\phi_\alpha$'s, see Remark \ref{rem: no phi_i zero}. Throughout the paper, we use $\nu$ for the Killing form on $\mathfrak{g}$ as well as its dualization on $\mathfrak{g}^*$. 
    \begin{thma}\label{thm: firstthm}
    
        \par  Let $G$ be a simple complex Lie group of adjoint type. A solution to Hitchin's equations for a stable Coxeter cyclic $G$-Higgs bundle over a closed Riemann surface $S$ with data $L_\alpha,\phi_\alpha$ as above is equivalent to a collection of Hermitian metrics $\mu_\alpha$ on $L_\alpha$ such that $\Theta_*\prod_{\alpha\in\mathcal{Z}}\mu_\alpha^{n_\alpha}$ is the flat metric on $\mathcal{O}$, and such that 
        \begin{align}\label{eq: firstthm}
            F(\mu_\alpha)+\sum_{\beta\in\mathcal{Z}} {\nu(\alpha,\beta)} \mu_\beta(\phi_\beta\wedge\bar{\phi}_\beta)=0,
        \end{align} 
        where $F(\mu_\alpha)$ denotes the curvature of the Chern connection associated with the Hermitian metric $\mu_\alpha$, which appears in the second term of (\ref{eq: firstthm}) as a map $L_\alpha\otimes\bar{L}_\alpha\to\mathbb{C}$. 
    \end{thma} 
    Taking $G^{\R}$ to be a split real form of $G,$ $(P,\phi)$ might carry the extra structure of a $G^{\R}$-Higgs bundle (see \S \ref{sec: introducing harmonic maps} and \S \ref{sec: G^R higgs} for definitions and discussion). In this case, Proposition \ref{prop: structure} and Theorem A can be refined. See \S \ref{sec: G^R higgs}. As well, we point out in Remark \ref{rem: not closed} that, after adding a hypothesis, the assumptions on $S$ being closed and $\phi$ being stable can be removed.
    
For $G$-Higgs bundles in the Hitchin section, the local coordinate expression of (\ref{eq: firstthm}) is similar to the expression for Hitchin's equations from Baraglia's work \cite{Baraglia2010CyclicHB} (which he calls affine Toda equations). To prove Theorem A in our general setting, we draw on \cite{Baraglia2010CyclicHB}, \cite{C}, and \cite{GP}.  One conclusion of Theorem A is that Hitchin's equations can be understood through the Cartan matrix or the extended Dynkin diagram of $\mathcal{Z}$, which carry the information of the $\nu(\alpha,\beta)$'s.

For $G=\textrm{SL}(n,\C)$, the system (\ref{eq: firstthm}) appears frequently in the works of Dai and Li and their collaborators (for example, \cite{CL}, \cite{DL}, \cite{Dai2018}, \cite{LM1}), although it usually looks different because the authors work with associated vector bundles and not $G$-bundles. With this in mind, we use Theorem A to prove results on harmonic maps that fall in line with those of \cite{DL} and \cite{Dai2018}, but for general Lie groups.

\begin{remark}
   For $G=\textrm{SL}(n,\C),$ a Coxeter cyclic $G$-Higgs bundle is equivalent to a certain twisted quiver bundle, as considered in \cite{CLGP}. As explained in \cite[Remark 4.1.3]{C}, on a closed surface and under the appropriate stability conditions, the solution to Hitchin's equations is equivalent to the solution to the relevant twisted quiver $(\sigma,\tau)$-vortex equations found in \cite{CLGP}. For general $G$, it would be interesting to compare our equations and solutions with equations and solutions from \cite{CLGP}. 
\end{remark}
\begin{remark}
    As we've mentioned, cyclic Higgs bundles correspond to other special immersions of surfaces in homogeneous spaces of $G$ (see also Remark \ref{rem: harmonic lift}). It appears that the equations governing those immersions are always expressed in terms of the $\mu_\alpha(\phi_\alpha\wedge\overline{\phi}_\alpha)$'s. For example, in \cite{Nie}, Nie studies alternating surfaces in pseudo-hyperbolic spaces, which naturally correspond to certain Coxeter cyclic $\textrm{SO}(2n+1,\C)$-Higgs bundles. For such a surface, the energy density is completely encoded by just one of the $2$-forms $\mu_\alpha(\phi_\alpha\wedge\overline{\phi}_\alpha)$ (see Theorems 4.6 and 4.10 in \cite{Nie}, it is the ``$h_1$-term"). Thus, the methods from this paper should be applicable in studying these other types of immersions.
\end{remark}
\begin{remark}
    After writing the first version of this paper, we learned of McIntosh's paper \cite{Mcintosh}. In \cite{Mcintosh}, among other things, McIntosh studies so-called $\tau$-primitive harmonic maps to $G/T$, where $T$ is a compact torus in $G$ fixed pointwise by a Coxeter automorphism $\tau$ of $G$, and shows that they are equivalent to solutions to what he calls the geometric Toda equations. The harmonic maps from Theorem A lift to $\tau$-primitive harmonic maps, where $\tau$ is an inner Coxeter automorphism (see Remark \ref{rem: harmonic lift}), and our equation (\ref{eq: firstthm}) is more or less equivalent to the geometric Toda equations. 

McIntosh explains how a $\tau$-primitive harmonic map gives rise to a $G$-Higgs bundle, but the structure of the $G$-Higgs bundle is not his main focus. Our work and \cite{Mcintosh} complement each other quite nicely--we work with different but related objects (for McIntosh, maps to $G/T$, and for us, solutions to Hitchin's equations), and it would be interesting to compare features of the two objects that come into play here and in \cite{Mcintosh} respectively. There is one part of the computation in \S \ref{sec: equation} that resembles a step in the derivation of the equations in \cite[\S 4]{Mcintosh}, but we make no changes. We also point out that the proof of Theorem A contains more value than just the statement: many aspects apply to general (not Coxeter) cyclic $G$-Higgs bundles (which do not lift to $\tau$-primitive harmonic maps), and it is mostly at the very end that specializing to $m=r$ is necessary.
\end{remark}

\subsection{Energy density}\label{subsec: energy density intro}
We use $\nu$ for the Killing metric on $G/K$, and fix a conformal metric $\mu$ on $S$, whose lift to $\tilde{S}$ we continue to denote by $\mu$. Let $(P,\phi)$ be a stable Coxeter cyclic $G$-Higgs bundle. For $t\in \C^*,$ each $G$-Higgs bundle $(P,t\phi)$ is stable and Coxeter cyclic. We say that $(P,\phi)$ is fixed by the $\C^*$-action if $\phi$ is gauge equivalent to $t\phi$ for all $t\in \C^*$.

The energy density of a map $f:\tilde{S}\to G/K$ is a norm for the derivative $df$, defined using $\mu$ and $\nu$ and denoted $e(f).$ Formally, $e(f)=\textrm{tr}_\mu f^*\nu.$ It descends to a function on $S,$ which vanishes precisely on the zero set of the Higgs field of $f$. Using Proposition \ref{prop: structure} and Theorem A, $e(f)$ splits as a sum of functions that satisfy an elliptic system depending on $G$, and which we name the Bochner-Toda equations (see Proposition \ref{prop: energy} and Theorem \ref{thm: BT eqns}). Using this system, we prove the following.
\begin{thmb}\label{thm: energy monotonicity}
Let $G$ be a simple complex Lie group. Let $(P,\phi)$ be a stable and simple Coxeter cyclic $G$-Higgs bundle on a closed Riemann surface $S$, which is not a fixed point of the $\C^*$-action. For each $t\in \C^*$, let $f_t:\tilde{S}\to G/K$ be the equivariant harmonic map associated with $(P,t\phi)$. Then, away from the zero set of $\phi$, the $\C^*$-family of functions $$e(f_{t}): S\to [0,\infty), \hspace{2mm} t\in \C^*,$$ is strictly increasing with $|t|$.
\end{thmb}

Dai-Li conjectured the result to hold for all stable Higgs bundles \cite{DL}. In \cite[Theorem 1.1]{Dai2018}, Dai-Li proved the $G=\textrm{SL}(n,\C)$ case of Theorem B. To do so, they developed a new maximum principle for elliptic systems \cite[Lemma 3.1]{Dai2018}. For general $G,$ even though the Bochner-Toda equations depend on the combinatorics of different Dynkin diagrams, we find that Dai-Li's maximum principle is well-suited to studying the system, and we apply it in the key step of the proof of Theorem B. 

While working on this paper, we found a generalization of Dai-Li's maximum principal, which has a combinatorial flavour and gives a new perspective. For pedagogical purposes, and to make the current paper more self-contained, we state and prove our new maximum principle in the appendix. The main results there are Theorems \ref{thm:mp} and \ref{thm:dai-li-gen}. We hope they might be of use in the future. 

Theorem B can be applied to study volume entropy. Let $g$ be a Riemannian metric on $S,$ which is allowed to have isolated degenerate points, and let $g$ as well denote the lift to $\tilde{S}$. Fixing a basepoint $p_0\in \tilde{S}$, let $\pi_1(S)\cdot p_0$ denote the orbit of $p_0$ under the $\pi_(S)$ action on $\tilde{S}$ by deck transformations, and let $B_g(p_0,R)$ denote a ball of radius $R$ around $p_0$ with respect to the distance function of $g$.
\begin{defn}
    The volume entropy of $g$ is $$\textrm{Ent}(g) = \limsup_{r\to\infty}\frac{\log \#\{p\in \tilde{S}:p\in\pi_1(S)\cdot p_0\cap B_{g}(p_0,R)\}}{r}.$$
\end{defn}
We show in Proposition \ref{prop: entropy} that for any stable family $(P,t\phi),$ $t\in \C^*,$ such that the harmonic maps $f_t$ are weakly conformal, if $\phi$ is not nilpotent, then $\lim_{|t|\to\infty}\textrm{Ent}(f_t^*\nu) =0.$ Using Theorem B, we see that in the Coxeter cyclic case, this convergence to zero is monotonic.
\begin{corb}\label{corollary}
   Let $(P,t\phi)$, $t\in \C^*$, be a family of stable and simple Coxeter cyclic $G$-Higgs bundles that are not $\C^*$-fixed points, with harmonic map $f_t$.  Then, as $|t|\to\infty,$ $\textrm{Ent}(f_t^*\nu)$ decreases monotonically to zero.
\end{corb}
Analogous results have been proved for other objects that arise from cyclic Higgs bundles, such as hyperbolic affine spheres in $\mathbb{R}^3$ \cite{Nieconvex} and maximal surfaces in $\mathbb{H}^{2,2}$ \cite{OT}. Since the pullback metrics for these objects can be expressed using the $\mu_\alpha(\phi_\alpha\wedge \overline{\phi}_\alpha)$'s, our analysis (Lemma \ref{lm:ineq-alpha} specifically) can be used to recover these results.

Finally, the $\C^*$-action on the set of stable $G$-Higgs bundles does not extend to a $\C$-action. However, using properness of the Hitchin fibration (see below), we can, in some sense, take $t\to 0$ in Theorem B and Corollary B. As a consequence, we obtain certain domination results, which generalize \cite[Theorem 1.3]{Dai2018}. See \S \ref{subsec: t to zero} for discussion and results.

\subsection{Hitchin representations and extrinsic curvature}
The Hitchin fibration from the moduli space of $G$-Higgs bundles on $S,$ $\mathcal{M}_S(G)$, to the Hitchin base for $G,$ $\oplus_{i=1}^l H^0(S,\mathcal{K}^{m_i})$ ($m_i$ depending on $G$), is an important integrable system. For $S$ a closed Riemann surface of genus $g\geq 2$, Hitchin constructed a section, and the image of that section under the non-abelian Hodge correspondence is a connected component of the $G^{\R}$-character variety, where $G^{\R}$ is a split real form of $G$, called the Hitchin component. The $G$-Higgs bundles $(P,\phi)$ in question are $G^{\R}$-Higgs bundles (see \S \ref{sec: introducing harmonic maps} and \ref{sec: G^R higgs}), and the equivariant harmonic maps land in a copy of the $G^{\R}$-symmetric space. See \S \ref{sec: hitchin section} for the explicit Hitchin section. Hitchin components are the first examples of higher Teichm{\"u}ller spaces, and the representations that they parametrize, so-called Hitchin representations, possess many interesting geometric and dynamical properties. For an overview on higher Teichm{\"u}ller theory and Hitchin representations, see \cite{Wie}. 

See \cite{Li} for a list of conjectures about harmonic maps for Hitchin representations. In particular, in \cite{DL} and \cite{Li} it is conjectured that harmonic maps for Hitchin representations to $\textrm{SL}(n,\R)$, at immersed points, are never tangent to flats in the symmetric space. In \cite{DL}, this result is proved for the Hitchin harmonic maps associated with Coxeter cyclic $\textrm{SL}(n,\C)$-Higgs bundles. Our main result on the geometry of harmonic maps is the extension of this result for general $G.$ Given an equivariant immersion $f:\tilde{S}\to G/K,$ we write $K_{\nu}(f_* T\tilde{S})$ for the intrinsic curvature function, a priori living on $\tilde{S}$, but which descends to $S$.
\begin{thmc}\label{thm: curvature}
      Let $G$ be a split real form of a simple complex Lie group distinct from $\mathrm{E}_7, \mathrm{E}_8$, and let $(P,\phi)$ be a Coxeter cyclic $G$-Higgs bundle over a closed Riemann surface $S$ in the Hitchin section with equivariant harmonic map $f:\tilde{S}\to G/K$. Then, at every point of $S,$
    \begin{align*}
        K_{\nu}(f_*T\tilde{S})<0.
    \end{align*}
\end{thmc} 
The type $\textrm{A}_n$ case is the result from \cite{DL}, and as we'll explain in \S \ref{sec: previous results}, the result is more or less proved for Lie groups of type $\textrm{B}_n$, $\textrm{C}_n$, and $\textrm{G}_2$ in \cite{DL}, but not stated as such. 

The proof of Theorem C consists of studying a system of equations derived from the Bochner-Toda equations (hence, from Theorem A), and then finding the right way to apply Dai-Li's maximum principle.  The proofs depend on the combinatorics of the extended Dynkin diagrams. For $\textrm{B}_n$ and $\textrm{D}_n,$ we find a quick proof that uses the ``prong" structure on the left end of the extended Dynkin diagram.
For cases $\textrm{C}_n, \textrm{F}_4,$ and $\textrm{G}_2,$ our argument makes strong use of the fact that the Dynkin diagram of $\mathcal{Z}$ is an undirected path. For cases $\textrm{A}_n$ and $\textrm{E}_6$, the Dynkin diagram of $\mathcal{Z}$ is not a path, but the $G^{\R}$-Higgs bundle condition imposes an extra symmetry that reduces our equations to a set of equations indexed by a folding of the Dynkin diagram, which is a path. The cases $\textrm{A}_n,$ $\textrm{C}_n$, $\textrm{E}_6,$ $\textrm{F}_4,$ and $\textrm{G}_2$ can then be treated all together. The $\textrm{E}_7$ and $\textrm{E}_8$ cases involve different combinatorics, which we don't pursue here.

Hitchin representations are defined for simple Lie groups; we include a discussion on the semisimple case in Remarks \ref{rem: products1} and \ref{rem: products2}. About semisimple groups, we find that we always have negative extrinsic curvature for products of Fuchsian representations into $\textrm{PSL}(2,\R)^n$ (for further comments see Remark \ref{rem: products2}). A component of such representations in the $\textrm{PSL}(2,\R)^n$ character variety can be seen as a cousin of the Hitchin component, so it's encouraging to know that the analogue of Dai-Li's conjecture holds for all such representations.
\begin{remark}
    In fact, Hitchin representations are defined in \cite{Hi} for groups of adjoint type. Except for rank $1$ groups, every Hitchin representation lifts to the split real form of the universal cover (see \cite{Hi}), so we're calling these lifts Hitchin as well.
\end{remark}

Finally, for the cases $\textrm{A}_n,$ $\textrm{C}_n$, $\textrm{E}_6$,  $\textrm{F}_4$, and $\textrm{G}_2$, we in fact deduce Theorem C from a more general result. Let $(P,\phi)$ be a simple and stable Coxeter cyclic $G$-Higgs bundle over a closed surface $S$ for $G$ of type $\textrm{C}_n$, $\textrm{F}_4,$ or $\textrm{G}_2$, or a simple and stable Coxeter cyclic $G^{\R}$-Higgs bundle for $G$ of type $\textrm{A}_n$ or $\textrm{E}_6.$ We assume that $(P,\phi)$ is not a $\C^*$-fixed point. Our work in \S \ref{sec: curvature} shows that for every $G$-Higgs bundle of the type above, there is an associated ordered set of divisors $D_0,\dots, D_\ell$, $\ell\leq |\mathcal{Z}|$. If $G$ is of type $\textrm{C}_n$, $\textrm{F}_4,$ and $\textrm{G}_2$, then these divisors are just the divisors of the holomorphic sections $\{\phi_\alpha\}_{\alpha\in\mathcal{Z}}$ from Proposition \ref{prop: structure}, ordered according to their positions from left to right on the extended Dynkin diagram. If $G$ is type $\textrm{A}_n$ or $\textrm{E}_6$, then these divisors are also divisors coming from $\{\phi_\alpha\}_{\alpha\in\mathcal{Z}}$, except we’ve cut out certain repetitions, and the ordering is done according to a different Dynkin diagram (which is a path, see \S \ref{sec: thms C and C'} for details). Viewing divisors as functions from $S\to \mathbb{Z}$, we write $D_i\leq D_j$ to mean that $D_i(x)\leq D_j(x)$ at every point $x\in S$. We write $D_i<D_j$ to specify futher that $D_i(x)<D_j(x)$ at every point $x$ in the support of $D_j$.  In the case of the Hitchin section, $D_0$ is the divisor of a holomorphic differential on $S$, and $D_1=\dots = D_{\ell}=0.$ 
\begin{thmcp}\label{thm: curvature extended version}
    For $G$ of type $\textrm{A}_n,\textrm{C}_n,\textrm{E}_6,\textrm{F}_4,$ or $\textrm{G}_2$, let $(P,\phi)$ be as above, with ordered set of divisors $D_0,\dots, D_{\ell}$, and equivariant harmonic map $f:\tilde{S}\to G/K$. Assume that 
    \begin{equation}\label{eq: divisor domination}
        D_0>D_1\geq D_2\geq \dots \geq D_{\ell}=0.
        \end{equation}
Then, at every point of $S$, $$K_{\nu}(f_*T\tilde{S})<0.$$
\end{thmcp}
See Remark \ref{rem: BDextension} for a similar extension of Theorem C that applies to groups of type $\textrm{B}_n$ and $\textrm{D}_n.$ For completeness, we also explain that an analogous negative curvature result holds for $\C^*$-fixed points (see the end of \S \ref{subsec: curvature}).

Nie in \cite{Nie} and Collier-Toulisse in \cite{CB} impose conditions similar to (\ref{eq: divisor domination}) on the holomorphic data for alternating surfaces, showing that when the condition holds, the alternating surfaces are infinitesimally rigid. In \cite{Ev}, Evans utilizes inequalities that stem from conditions like (\ref{eq: divisor domination}) in his construction of domains of discontinuity for $\textrm{G}_2$-Hitchin representations. It thus seems reasonable to ask whether conditions like (\ref{eq: divisor domination}) are geometrically significant. We pose the following. 

\begin{que}\label{que: anosov}
    Given a $G$-Higgs bundle as in Theorem C', let $\rho:\pi_1(S)\to G$ be the representation obtained via the non-abelian Hodge correspondence. Is $\rho$ Anosov?
\end{que}

\subsection{Acknowledgements} We are grateful to Brian Collier, Qiongling Li, Oscar Garc{\'i}a-Prada, and J{\'e}r{\'e}my Toulisse for helpful discussions. We also thank the referees for sharing many thoughtful suggestions and remarks. At the time of writing and submission, N.S. was funded by the FNR grant O20/14766753, \textit{Convex Surfaces in Hyperbolic Geometry.}

\section{Preliminaries}
 Throughout this section, let $G$ be a semisimple (real) Lie group with no compact factors and Lie algebra $\mathfrak{g}$. We fix a $G$-invariant non-degenerate bilinear form $\nu$ on $\g$. When $G$ is simple, $\nu$ is a multiple of the Killing form, but in the semisimple case there is a choice of scaling on each simple factor. In any case, after this section, most of this paper will be concerned with simple Lie groups. 

\subsection{Homogeneous and symmetric spaces} 
Given a subgroup $H<G$ and a space $V$, an action $\alpha$ of $H$ on $V$ determines an equivalence relation via $(g,v)\sim (hg,\alpha(h)v)$ for all $h\in H.$ The associated bundle is $G\times_{\alpha} V=(G\times V)/\sim.$ When the action $\alpha$ is implicit, we might write $G\times_H V$. If $H$ preserves a structure on $V$ (vector space structure, Lie bracket, etc.) then the associated bundle inherits that structure.

For $g\in G$, let $L_g:G\to G$ be the left multiplication map. The Maurer-Cartan form of $G$ is the section $\omega$ of $\Omega^1(G,\g)$ defined on vectors $v\in T_gG$ by $$\omega(v) = (L_{g^{-1}})_*v.$$
The Maurer-Cartan form provides an isomorphism of vector bundles from $TG$ to $G\times \g$, and transports $\nu$ on $\g$ to a pseudo-Riemannian metric on $G$, which we continue to denote by $\nu$.

 Let $(M,\sigma)$ be a pseudo-Riemannian manifold with a $C^\infty$ and transitive action of $G$. If we fix a point $x\in M$ and set $H=\textrm{Stab}_G(x)$, $\mathfrak{h}=\textrm{Lie}(H)$, then $M$ identifies with $G/H$ and the Maurer-Cartan form yields an isomorphism between $TM$ and the associated bundle $G\times_H \g/\mathfrak{h}$. We will be interested in the case where $M$ is a reductive homogeneous space, which means that under the adjoint action of $H$, $\g$ admits a decomposition into $\textrm{ad}_H$-modules, $\g=\mathfrak{h}\oplus \mathfrak{m}.$ In this setting, $TM$ identifies with $M\times_{\textrm{ad}|_H} \mathfrak{m},$ and $\sigma$ is realized under this identification as coming from a non-degenerate bilinear form of $\g$ (which, when $G$ is simple, is a scalar multiple of $\nu$). Conversely, if $G/H$ is reductive, then any $\nu$ as above determines a $G$-invariant pseudo-Riemannian metric on $G/H.$
 
 The most important case under consideration here is when $H$ is a maximal compact subgroup $G$. To distinguish this case, we will write $K$ for maximal compact subgroups. Then (after possibly multiplying $\sigma$ by $-1$) the space $(M,\sigma$) is a Riemannian symmetric space of non-compact type. We always write the associated Lie algebra splitting, i.e., the Cartan decomposition, as $\g=\kk\oplus \p,$ with $\kk$ the Lie algebra of $K$, and $\p$ the $\nu$-orthogonal complement. When $G$ is simple, since the metric is unique up to scale, we will always assume it comes from the Killing form, and refer to this space as ``the symmetric space" of $G$. We might write just $G/K$ when we mean $(G/K,\nu).$

\subsection{Harmonic maps and Higgs bundles}\label{sec: introducing harmonic maps}
Let $(M,\sigma)$ be a pseudo-Riemannian manifold and let $\rho$ be an action of $\pi_1(S)$ on $(M,\sigma)$ by isometries. From a $C^2$ $\rho$-equivariant map $f:\tilde{S}\to (M,\sigma)$ we form the pullback bundle $f^*TM$, and the derivative $df$ defines a section of the endomorphism bundle $T^*\tilde{S}\otimes f^*TM$. Complexifying, let $df=\partial f +\overline{\partial}f$ be the decomposition into $(1,0)$ and $(0,1)$ parts. We denote by $\nabla$ the connection on $f^*TM\otimes \C$ induced by the Levi-Civita connection of $\sigma$, and we also use $\nabla$ for its extension to $f^*TM\otimes \C$-valued forms.  
\begin{defn}
The map $f$ is harmonic if $\nabla^{0,1}\partial f=0.$
\end{defn}

Now, take the target to be the symmetric space $(G/K,\nu)$ and assume from now on that $G$ is a complex Lie group. We will explain that an equivariant harmonic map to $(G/K,\nu)$ gives us the data of a $G$-Higgs bundle with a reduction of structure group satisfying the self-duality equations. Given a holomorphic $G$-bundle $P$ over $S,$ we denote the adjoint bundle $P \times_{\textrm{ad}}\g$ by $\textrm{ad}P$.
\begin{defn}
A $G$-Higgs bundle is a pair $(P,\phi),$ where $P$ is a principal $G$-bundle over $S$ and $\phi$ is a holomorphic section of $\textrm{ad}P\otimes \mathcal{K}$ called the Higgs field.
\end{defn}
Returning to harmonic maps, associated with $\rho$ we have the principal $G$-bundle $\tilde{S}\times_\rho G$ over $S$. The $\rho$-equivariant map $f$ is equivalent to a section of $\tilde{S}\times_\rho G$, which is in turn equivalent to a reduction of the structure group to $K$, say $Q\subset \tilde{S}\times_\rho G$. Using the isomorphism $TG/K=G/K\times_{\textrm{ad}|_K} \p,$ we see that $h^* TG/K$ identifies with $Q\times_{\textrm{ad}|_K}\p$. Since both $\p$ and $\kk$ complexify to $\g$, the complexification of $Q$ is a $G$-bundle, say $P$, and $f^*TG/K\otimes_{\R}\C= P \times_{\textrm{ad}}\g= \textrm{ad}P.$ Under this last identification, via the Koszul-Malgrange theorem we endow $P$ with the complex structure coming from the $(0,1)$-part of the Levi-Civita connection for $G/K.$ Moreover, $\partial f$ becomes a $(1,0)$-form valued in $\textrm{ad}P$, say $\phi,$ and the harmonic map equation expresses that $\phi$ is holomorphic. Taking stock, the harmonic map $f$ gives a $G$-Higgs bundle $(P,\phi).$  

\par Given any $G$-Higgs bundle $(P,\phi)$ (not necessarily the same as above), a reduction of the structure group to $K$ gives us an antilinear involution $\rho_h:\textrm{ad}P\to \textrm{ad}P,$ which acts as a Cartan involution on each fiber. Furthermore, $\rho_h$ determines a Hermitian metric $h$ on $\textrm{ad}P$ via $h(x,y)=\nu(x,-\rho_h(y)),$ and $-\rho_h$ defines a real structure on each fiber. As is customary, we denote $-\rho_h(x)=x^{*_h}$.
\begin{defn}
    Given a $G$-Higgs bundle $(P,\phi)$, we say that the reduction satisfies Hitchin's self-duality equations if $\nabla_h + \phi +\phi^{*_h}$ is flat, or, in other words, if 
    \begin{equation}\label{SDeqns}
        F(\nabla_h) + [\phi,\phi^{*_h}]=0,
    \end{equation}
    where $\nabla_h$ is the Chern connection on $P$ associated with $h$.
\end{defn}

Now, for our harmonic map $f:\tilde{S}\to (G/K,\nu)$ and its Higgs bundle $(P,\phi),$ under the identification $f^*TG/K\otimes_{\R}\C=\textrm{ad}P$, the pullback of the Levi-Civita connection of $\nu$ is in fact the Chern connection $\nabla_h$ on $P$ \cite{D}. As well, $\nabla_h +\phi+\phi^{*_h}$ is the natural flat connection. In other words, the map $f$ also encodes a solution to the self-duality equations.

A gauge transformation of a $G$-bundle $P$ is a smooth $G$-bundle isomorphism, i.e., a $G$-equivariant diffeomorphism from $P\to P$ that covers the identity. Any gauge transformation also induces an automorphism of $\textrm{ad}P$ that in each fiber in a trivialization is realized as an inner automorphism of $\g$. Given $P$, we will always denote the group of gauge transformations by $\mathcal{G}$. The gauge group also acts on the set of complex structures of $P$; given $g\in \mathcal{G}$, we write $g^*P$ for the smooth bundle $P$ equipped with the new complex structure. If $(P,\phi)$ is a $G$-Higgs bundle with a reduction of structure group inducing an involution $\rho_h$ solving (\ref{SDeqns}), and $g\in \mathcal{G}$, then $g^*\rho_h$ solves (\ref{SDeqns}) for $(g^*P,g^*\phi)$.
\begin{remark}\label{rem: S^1 invariance}
    If $(P,\phi)$ admits a reduction leading to a solution to (\ref{SDeqns}), then for any unit norm $\zeta\in \C^*,$ that same reduction solves the self-duality equations for $(P,\zeta\phi).$
\end{remark}
The two foundational existence results below constitute the non-abelian Hodge correspondence for complex Lie groups. A representation $\rho:\pi_1(S)\to G$ is irreducible if it is not contained in a parabolic subgroup, and $\rho$ is reductive if its Zariski closure is reductive. The following is due to Donaldson \cite{D} and Corlette \cite{Co}.

\begin{thm}\label{thm: donaldson, corlette}
Let $\rho:\pi_1(S)\to G$ be a reductive representation. There exists a $\rho$-equivariant harmonic map $f:\tilde{S}\to G/K.$ If $\rho$ is irreducible then $f$ is unique.
\end{thm}
There are notions of stability and polystability for a $G$-Higgs bundle $(P,\phi)$, defined in terms of an antidominant character for a parabolic subgroup of $G$ and a holomorphic reduction of the structure group of $P$. We need these notions for the next result, but we omit the definitions since we won't be making any use of them. Roughly speaking, we denote the moduli space of gauge equivalence classes of $G$-Higgs bundles by $\mathcal{M}_S(G)$ (see \cite{Fan} for the construction). A $G$-Higgs bundle on $S$ is simple if its stabilizer in the gauge group identifies in each fiber with the center of $G$ intersected with the kernel of the adjoint representation, and if $(P,\phi)$ is stable and simple, and stable as a $(G\times \overline{G})$-Higgs bundle, then $(P,\phi)$ defines a smooth point of $\mathcal{M}_S(G)$. See \cite{G-P} for the definitions and theory around stability and simplicity. The following is proved in \cite{S2} and \cite{BGM}.
\begin{thm}
   A $G$-Higgs bundle is polystable if and only if there exists a reduction of the structure group satisfying (\ref{SDeqns}). If it is stable, then the reduction is unique.
\end{thm}

Higgs bundles can be defined for real Lie groups, and a lot of the discussion above goes through with minor changes for real Lie groups, but in this paper we're concerned mainly with complex Lie groups. Formally, let $G_0$ be real and let $K$ be a maximal compact subgroup of $G_0$ with complexification $K^{\C}$. Setting $\g_0$ and $\mathfrak{k}$ to be the Lie algebras of $G_0$ and $K$ respectively, there is a 
Killing-orthogonal splitting $\g_0=\mathfrak{k}\oplus \p$. A $G_0$-Higgs bundle is a pair $(P_{K^{\C}},\phi),$ where $P_{K^{\C}}$ is a principal $K^{\C}$-bundle and $\phi$ is a holomorphic $1$-form valued in $\textrm{ad}\p^{\C}:=P_{K^{\C}}\times_{\textrm{Ad}|_{K^{\C}}} \p^{\C}$. A motivation for the definition is that a harmonic map to the symmetric space of $G_0$, similar to the procedure above, gives rise to a $G_0$-Higgs bundle. Observe that, taking $G$ to be the complexification of $G_0,$ by extending the  structure group to $G$ via $P:=P_{K^{\C}}\times_{K^{\C}} G,$ we can associate the $G$-Higgs bundle $(P,\phi)$. Later in this paper, we will take the perspective that $G_0$-Higgs bundles are $G$-Higgs bundles with some extra structure. 

Finally, while we're working in the setting of $G$-Higgs bundles in this paper, we'll encounter ordinary Higgs bundles at a few points. A Higgs bundle in this sense is a pair $(\mathcal{E},\phi)$ consisting of a holomorphic vector bundle $\mathcal{E}$ over $S$ together with holomorphic element of $\Omega^{1,0}(\textrm{End}(\mathcal{E}))$ called the Higgs field. Given a $G$-Higgs bundle $(P,\phi)$, a complex vector space $V,$ and a 
linear action $\sigma: G\to \textrm{GL}(V)$, we obtain an ordinary Higgs bundle as follows: the associated vector bundle $\mathcal{E}=P\times_\sigma V$ carries a holomorphic structure, and the induced action of $\phi$ on $\mathcal{E}$ defines a Higgs field. On the other hand, let $\mathcal{E}$ be a holomorphic vector bundle of rank $n$ with structure group $G,$ and let $\phi$ be a Higgs field on $\mathcal{E}$. The frame bundle $\mathcal{F}(\mathcal{E})$ of $\mathcal{E}$ is a $\textrm{GL}(n,\C)$-principal bundle, and via the natural isomorphism between the adjoint bundle $\textrm{ad}\mathcal{F}(\mathcal{E})$ and $\textrm{End}(\mathcal{E})$, $\phi$ gives rise to a holomorphic $\textrm{ad}\mathcal{F}(\mathcal{E})$-valued $(1,0)$-form, say $\mathcal{F}\phi$. Since the structure group of $\mathcal{E}$ reduces to $G$, we can reduce $\mathcal{F}(\mathcal{E})$ to a holomorphic principal $G$-bundle $P$, and if $\mathcal{F}\phi$ lives in the image of the inclusion $\textrm{ad}P\otimes \mathcal{K}\to \textrm{ad}\mathcal{F}(\mathcal{E})\otimes \mathcal{K},$ then we see that $(P,\mathcal{F}\phi)$ defines a $G$-Higgs bundle.

\section{Cyclic $G$-Higgs bundles and affine Toda equations}
Let $G$ be a complex simple Lie group with Lie algebra $\mathfrak{g}$ and Killing form $\nu$. In this section we use Lie theory to study cyclic $G$-Higgs bundles and the accompanying set of Hitchin's equations. After giving background on Lie theory and introducing (Coxeter) cyclic $G$-Higgs bundles, we prove Proposition \ref{prop: structure} and Theorem A in \S \ref{sec: cyclic definition} and \S \ref{sec: equation}. We then explain the case of $G^{\R}$-Higgs bundles in \S 3.5, and we recall the main example, that of $G$-Higgs bundles in the Hitchin section, in \S 3.6. Our main reference for the Lie theory is \cite[Chapter 3]{OV}. For related work, see \cite{Baraglia2010CyclicHB}, \cite{B}, \cite{C}, \cite{GP}, \cite{GPRi}, and \cite{Mcintosh}.
.

\subsection{Root systems and admissible systems}\label{sec: root systems}
Let $(V,\langle\cdot,\cdot\rangle)$ be an inner product space. For any non-zero vector $\alpha\in V,$ we have an orthogonal hyperplane $L_\alpha$ and reflection in $L_\alpha$, say $r_\alpha:V\to V$. A finite subset $\Delta$ of $V$ is a root system if 
\begin{itemize}
    \item for every $\alpha\in V$, $r_\alpha(\Delta)=\Delta,$ and 
    \item for every $\alpha,\beta\in V$, $2\frac{\langle \alpha,\beta\rangle}{\langle \beta,\beta\rangle} \in \mathbb{Z}$.
\end{itemize}
Given $\Delta,$ we can choose a set of positive roots $\Delta^+$, by which we mean
a subset such that
\begin{itemize}
    \item for each $\alpha\in \Delta,$ exactly one of $\alpha$ and $-\alpha$ is contained in $\Delta^+$, and
    \item if $\alpha,\beta\in \Delta^+$ and $\alpha+\beta$ is a root, then $\alpha+\beta\in \Delta^+.$
\end{itemize}
The choice of positive roots determines a subset $\Pi$, the set of simple roots, which is characterized by being linearly independent and by the property that every element in $\Delta^+$ is a non-negative linear combination of elements in $\Pi$. 

The simple roots of a root system are an example of an admissible system: a collection of vectors $\Gamma=\{v_1,\dots, v_n\}$ in $(V,\langle\cdot,\cdot\rangle)$ such that for every $i\neq j$, $a_{ij}:=2\frac{\langle v_i,v_j\rangle}{\langle v_j,v_j\rangle}$ is a non-positive integer \cite[Chapter 3, \S 1.7]{OV}. Out of an admissible system $\Gamma$, we can form a Cartan matrix $A(\Gamma)=(a_{ij})_{i,j=1}^n$, and a Dynkin diagram via the usual rules (see \cite[Chapter 3, \S 1.7]{OV}). An automorphism of $\Gamma$ is a permutation that preserves the inner products $\langle v_i, v_j\rangle.$ Given $\Gamma,$ we write the automorphism group as $\textrm{Aut}(\Gamma)$.

The central examples of root systems and admissible systems, and the ones most important for this paper, are those coming from Lie algebras. A Cartan subalgebra $\mathfrak{h}$ of $\g$ is a maximal abelian semisimple subalgebra. Simultaneously diagonalizing the adjoint action of $\mathfrak{h}$ on $\g$ yields the root space decomposition $$\g=\oplus_{\alpha} \g_{\alpha},$$ where $\alpha\in \mathfrak{h}^*$ is a root, i.e., a weight of the adjoint representation restricted to $\mathfrak{h}$, and $\g_{\alpha}$ is the root space $$\g_{\alpha}=\{x\in\g^{\C}: \textrm{ for all } h\in \mathfrak{h}, [h,x]=\alpha(h)x\}.$$
 The roots of the Lie algebra form a root system $\Delta$ on $(\mathfrak{h}^*,\nu)$. Choosing positive roots $\Delta^+$, the simple roots $\Pi$ are an admissible system.

Given a root $\beta = \sum_{\alpha \in \Pi} m_\alpha \alpha,$ $m_\alpha\in\mathbb{Z}$, the height of $\beta$ is $\sum_{\alpha \in \Pi} m_\alpha.$ The highest root $\delta$ is the root with the largest height. We will write $\delta=\sum_{\alpha \in \Pi} n_\alpha \alpha$. 
We define the extended simple root system $\mathcal{Z}$ to be $\Pi\cup\{-\delta\}$; $\mathcal{Z}$ is an admissible system itself, and its Dynkin diagram is called the extended Dynkin diagram of $\mathfrak{g}$. 

At a few points in this paper, it will be helpful to make use of a basis for $\g$ that's well adapted to the root system. For every root $\alpha$ we have the vector $h_\alpha\in \mathfrak{h}$, defined by the relation that, for all roots $\beta$, $\beta(h_\alpha) = 2\frac{\nu(\alpha,\beta)}{\nu(\alpha,\alpha)}$, where $\nu$ here is the dualized Killing form on $\g^*.$ One can choose root vectors $e_\alpha \in \g_\alpha$ so that we have a Chevalley basis for $\g$, $$\{e_\alpha,e_{-\alpha},h_\beta:\alpha\in \Delta^+, \beta\in \Pi\},$$ characterized by $[e_\alpha,e_{-\alpha}]=-h_\alpha$ and $[e_\alpha,e_\beta]=N_{\alpha,\beta}e_{\alpha+\beta}$, for some integers $N_{\alpha,\beta}$ (see \cite[Chapter 3, Theorem 1.19]{OV}). For later use, we take note of the following fact about $\mathcal{Z}$. 
\begin{lem}\label{lem: commutators}
    If $\alpha,\beta\in\mathcal{Z}$, then, in a basis as above, $[e_\alpha,e_{-\beta}]=-\delta_{\alpha\beta} h_\alpha.$
\end{lem}
\begin{proof}
Recall that $[\g_\alpha,\g_\beta]$ is non-zero if and only if $\alpha+\beta$ is a root. Let $\alpha$ and $\beta$ be distinct simple roots. If $\alpha-\beta$ is a root, then without loss of generality it is positive, which contradicts simplicity of $\beta$. For any positive root $\alpha$, $\alpha+\delta$ is not a root, since if it were its height would be greater than that of $\delta$.
\end{proof}
As well, it will be helpful in the proof of Theorem A to have a Cartan involution of $\g$ that preserves $\mathfrak{h}.$ With respect to the Chevalley base, we define $\hat{\rho}:\g\to\g$ by
\begin{equation}\label{def: rho_0}
    \hat{\rho}(h_\alpha)=-h_\alpha, \hspace{1mm} \hat{\rho}(e_{\alpha})=-e_{-\alpha}, \hspace{1mm} \hat{\rho}(e_{-\alpha}) = - e_{\alpha}.
\end{equation}
This is the same $\hat{\rho}$ that Baraglia works with in \cite{Baraglia2010CyclicHB}.

Finally, to study systems of equations indexed by admissible systems, which we'll do to prove Theorems B,C, and C', we need to recall the relevant facts about the angles between simple roots. For $(V,\langle\cdot,\cdot\rangle)$ an arbitrary inner product space with root system $\Delta$, and $\alpha,\beta\in \Delta,$ set $a_{\alpha,\beta}=2\frac{\langle\alpha,\beta\rangle}{\langle\alpha,\alpha\rangle}$. Proofs for the proposition below can be found in \cite[Chapter 3, Proposition 1.7 and Theorem 1.12]{OV}.
\begin{prop}\label{prop: angle relations}
    Let $\Delta$ be a root system on an inner product space as above, with positive root set $\Delta^+$ and simple roots $\Pi$.
    \begin{itemize}
        \item If $\alpha,\beta\in \Pi$, then $a_{\alpha,\beta}=2$ if $\alpha=\beta$ and $a_{\alpha,\beta}\leq 0$ for $\alpha\neq \beta.$
        \item The highest root $\delta$ is dominant: if $\alpha$ is a simple root, then $a_{\alpha,\delta}\geq 0$ (and hence $a_{\alpha,-\delta}\leq 0$). 
    \end{itemize}
\end{prop}
For root systems arising from Lie algebras, the angles between the roots are constrained, so that $a_{\alpha,\beta}=0,2,\sqrt{2},$ or $\frac{2}{\sqrt{3}}.$ These values can be read off of the Cartan matrix or the Dynkin diagram of the extended system of simple roots.

\subsection{Inner automorphisms of Lie algebras}\label{sec: aut} Here we review the known classification of inner semisimple automorphisms of simple Lie algebras, and then we discuss Coxeter automorphisms.

We continue with $G, \mathfrak{g}, \mathfrak{h},$ $\Pi,$ etc., as in \S \ref{sec: root systems}. Any element $x\in\mathfrak{h}$ is determined by the collection $\{\alpha(x): \alpha\in \Pi\}$. Setting $x_{-\delta} = 1-\delta(x)$ and $x_\alpha=\alpha(x)$ for $\alpha\in \Pi$ we call $\{x_\alpha: \alpha \in \mathcal{Z}\}$ the barycentric coordinates of $x$. For $n_{-\delta}=1$, the barycentric coordinates satisfy $\sum_{\alpha\in\mathcal{Z}} n_\alpha x_\alpha= 1.$ 

The adjoint group of $G$, say $\textrm{ad}G$, is the image of $G$ under the adjoint representation. The fundamental group $\pi_1(\textrm{ad}G)$ is naturally identified with a subgroup of $\textrm{Aut}(\mathcal{Z})$ that acts simply transitively on the roots $\alpha$ with $n_\alpha =1$ (see \cite[Chapter 3, \S 3.6]{OV}). Moreover, $\pi_1(\textrm{ad}G)$ acts on $\mathfrak{h}$ by permuting barycentric coordinates. The theorem below is the classification of inner semisimple automorphisms.
\begin{thm}[Theorem 3.11 in Chapter 3 of \cite{OV}]\label{thm: 3.3.11}
Any inner semisimple automorphism of $\mathfrak{g}$ is conjugate to an automorphism of the form $\textrm{ad}_{g_x}$, where ${g_x}=\textrm{exp}(2\pi ix)$, and $x\in \mathfrak{h}$ satisfies 1) $\textrm{Re}(x_\alpha)\geq 0$ for all $i$, and 2) if $\textrm{Re}(x_\alpha)=0$, then $\textrm{Im}(x_\alpha)\geq 0$. Two such automorphisms are conjugate if and only if the barycentric coordinates of the first can be obtained from those of the second by a permutation from $\pi_1(\textrm{ad}G).$
\end{thm}
Any finite order automorphism of $\g$ is semisimple, by consideration of the minimal polynomial. If $s$ is an inner automorphism of $\g$ of finite order $k,$ then the element $x\in \mathfrak{h}$ satisfies, in the barycentric coordinates, $x_\alpha = p_\alpha/k$, for some non-negative $p_\alpha\in \mathbb{Z}$ such that $\sum_{\alpha\in \mathcal{Z}} n_\alpha p_\alpha = k.$
\begin{defn}\label{def: coxeter}
The Coxeter number $r$ of the root system $\Delta$ is $\sum_{\alpha\in \mathcal{Z}} n_\alpha.$ 
\end{defn}
For exposition on the Coxeter number, see \cite[\S 3.16]{Humph}. It can be characterized, among other ways, as the order of a Coxeter element of the Weyl group (a product of all reflections coming from a choice of simple roots), or as the highest degree of a polynomial in a homogeneous generating set for the algebra of $\textrm{ad}(G)$-invariant polynomials on $\g$.
\begin{defn}
    We say that an inner automorphism $s$ of $\mathfrak{g}$ is a Coxeter automorphism if $s$ has order $r.$ An element of $G$ inducing $s$ is called a Coxeter element.
\end{defn}
The condition that $s$ is a Coxeter automorphism forces every $p_i$ to be equal to $1$. We deduce the following properties.
\begin{prop}\label{prop: coxeter properties}
    All Coxeter automorphisms of $\g$ are conjugate. For $s$ an order $m$ automorphism of $\g$, let $\g=\oplus_{i=0}^{m-1}\g_i$ be the eigenspace grading, $\g_j=\{x\in \g: s(x)=e^{\frac{2\pi i j}{m}}x\}$. Then, if $s$ is Coxeter,
    \begin{enumerate}
        \item $\g_0=\mathfrak{h}.$
        \item $\g_1$ is the sum of the extended simple root spaces.
    \end{enumerate}
\end{prop}
The first assertion is immediate from Theorem \ref{thm: 3.3.11}. To see that $\g_0=\mathfrak{h}$, we conjugate $s$ into the form of Theorem \ref{thm: 3.3.11}, $s= \textrm{exp}(2\pi i x)$. It is seen from the structure of $x$ that $s$ is regular in the sense that if $s'$ is another inner automorphism whose fixed-point subalgebra is $\g_0',$ then $\dim \g_0\leq \dim \g_0',$ from which the claim follows. And to isolate the $\g_1$ in the grading, we again use the explicit form of the $x$.

\subsection{Cyclic $G$-Higgs bundles}\label{sec: cyclic definition}
Let $P$ be a holomorphic $G$-bundle over a Riemann surface $S$ with gauge group $\mathcal{G}$ and adjoint bundle $\textrm{ad}P.$ An element $s\in \mathcal{G}$ induces an automorphism of $\textrm{ad}P,$ which in any fiber in any trivialization of $\textrm{ad}P$ is realized as an inner automorphism of $\g$, which we continue to denote by $s$. If $s$ has finite order, then by connectedness of $S$, since the barycentric expression for $s$ has discrete values, it only changes according to the action of $\pi_1(\textrm{ad}G)$ as we vary over the surface. Hence, setting $\zeta=e^{\frac{2\pi i}{m}}$, the action of $s$ in any fiber yields a $\mathbb{Z}_m$-grading of $\g$ as $\g=\oplus_{i=0}^{m-1} \g_i,$ where $\g_i$ is the $\zeta^i$-eigenspace (this is the grading from Proposition \ref{prop: coxeter properties}). Let $G_0\subset G$ be the connected Lie group obtained via exponentiating $\g_0\subset \g$, which is always reductive \cite[Chapter 3, Proposition 3.6]{OV}. Inspired by Vinberg's work \cite{Vin}, a pair $(G_0,\g_i)$ is typically called a Vinberg $\theta$-pair \cite{GP}. In turn, we have a $\mathbb{Z}_m$-grading of $\textrm{Aut}P$ into $\zeta^i$-eigenbundles, $$\textrm{ad}P=\oplus_{j=0}^{m-1} \mathcal{C}_j.$$
The Lie algebra subbundle $\mathcal{C}_0$ then determines a reduction of structure group of $P$ from $G$ to $G_0,$ and each $\mathcal{C}_j$ is an associated bundle of this $G_0$-bundle, modelled on $\mathfrak{g}_j$.
\begin{defn}\label{def: cyclic}
    We say that $(P,\phi)$ is a cyclic $G$-Higgs bundle if there exists a holomorphic gauge transformation $s$ of finite order $m$ such that $s^*\phi = \zeta \phi.$
\end{defn}
That is, in the notation above, $\phi \in \mathcal{C}_1$.
Following an observation of Vinberg in \cite{Vin}, if $X \in \g_k$, then we can define a new grading modulo $m'=\frac{m}{\textrm{gcd}(m,k)}$ by setting $\g_j'=\g_{jk},$ for which we then have $X\in \g_1'.$ Hence, if $\phi\in \mathcal{C}_j$, we can always assume up to modifications that $\phi\in\mathcal{C}_1,$ so the definition above encompasses $\phi\in \mathcal{C}_j.$
 In this paper we make the following definition.
\begin{defn}
    A cyclic $G$-Higgs bundle $(P,\phi)$ is Coxeter cyclic if the gauge transformation $s$ has order equal to the Coxeter number.
\end{defn}
    Compare with other definitions of cyclic $G$-Higgs bundles \cite{Baraglia2010CyclicHB}, \cite{Dai2018},  \cite{GP}, \cite{GPRi}, \cite{L2}, \cite{Katz}. In most references, cyclic Higgs bundles are simple and stable cyclic $G$-Higgs bundles on a closed Riemann surface (and the order $m$ is usually the Coxeter number). On a closed surface we have the moduli space structure $\mathcal{M}_S(G)$, and hence such $G$-Higgs bundles are characterized as fixed points for the $\mathbb{Z}_m$-action on  $\mathcal{M}_S(G)$ given by multiplying the Higgs field by $m^{th}$ roots of unity. 
    
    A more general definition is given in \cite[\S 3.2]{GP} (see also \cite{GPRi}). Given a general $\theta\in \textrm{Aut}(\g)$ and a principal $G$-bundle $P,$ one gets another principal $G$-bundle $\theta(P),$ defined by the same total space as $P$, but equipped with the $G$-action $p\cdot g=p\theta(g)$. As well, $\theta$ transports any Higgs field $\phi$ on $P$ to a Higgs field $\theta(\phi)$ on $\theta(P).$ A $G$-Higgs bundle $(P,\phi)$ is $\theta$-cyclic if there exists a $G$-Higgs bundle isomorphism between $(P,\phi)$ and $(\theta(P),\zeta\theta(\phi)).$ For $\theta$ inner, $P$ is isomorphic to $\theta(P),$ and this is equivalent to a cyclic $G$-Higgs bundle in our sense. For general $\theta$, simple and stable $\theta$-cyclic $G$-Higgs bundles on closed surfaces can be interpreted as fixed points of other $\mathbb{Z}_m$-actions on the moduli space (see \cite{GPRi}).

\begin{remark}\label{rem: theta cyclic version}
   There should be versions of Proposition \ref{prop: structure} and Theorem A for $\theta$-cyclic $G$-Higgs bundles, but we don't address this here. Associated with every $\theta\in\mathrm{Out}(\g)$ is a root system $\Delta^\theta$, called the real $\theta$-root system, which comes with its extended system of real $\theta$-roots $\mathcal{Z}^\theta$ and its own Coxeter number $r^\theta$ (see \cite[Chapter 3, \S 3.9]{OV}). There is an analogue of Theorem \ref{thm: 3.3.11}, namely \cite[Chapter 3, Theorem 3.16]{OV}, which describes any automorphism of $\g$ in the class of $\theta\in\textrm{Out}(\g)$ in terms of barycentric coordinates on a semisimple abelian subalgebra $\mathfrak{h}^\theta\subset \g$, indexed by $\mathcal{Z}^\theta$. Presumably these notions can be used to study Hitchin's equations, as we'll do below in the case that $\theta$ is trivial (equivalently, inner). We'll see in Remark \ref{rem: harmonic lift} that when $S$ is closed, the $\theta$ trivial case corresponds to $\tau$-primitive harmonic maps for $\tau$ inner. We expect that when $\theta$ is non-trivial, one should get more general $\tau$-primitive harmonic maps (studied in \cite{Mcintosh}).
\end{remark}
We now prove one side of the equivalence of Proposition \ref{prop: structure} (without the adjoint type assumption). The other side can also be proven without any further preparations, but we defer the proof to later on, since the notation in the proof will be helpful in the proof of Theorem A.
\begin{prop}\label{prop: one side structure}
    A Coxeter cyclic $G$-Higgs bundle $(P,\phi)$ determines a collection of line subbundles $\{L_\alpha : \alpha\in\mathcal{Z}\}$, holomorphic sections $\phi_\alpha \in H^0(S,L_\alpha\otimes \mathcal{K})$, and an isomorphism from $\otimes_{\alpha\in\mathcal{Z}}L_\alpha^{n_\alpha}$ to $\mathcal{O}$.
\end{prop}
 In the notation of the paragraph before Definition \ref{def: cyclic}, for a general cyclic $G$-Higgs bundle, $\mathcal{C}_1$ decomposes into $\mathcal{C}_0$-adjoint-invariant subbundles according to the decomposition of $\g_1$ into irreducible representations of $\g_0$, $\mathcal{C}_1 = \oplus_{i} \mathcal{C}_1^i$. Proposition \ref{prop: one side structure} clarifies this decomposition in the Coxeter case.
\begin{proof}
 For $(P,\phi)$ a Coxeter cyclic $G$-Higgs bundle, the subgroup $G_0$ is a Cartan subgroup $H$ of $G.$ Thus, fixing a positive root set for the Lie algebra $\mathfrak{h}$, which then yields an extended simple root set $\mathcal{Z}$, Proposition \ref{prop: coxeter properties} implies that the $\textrm{ad}|_H$-decomposition of $\mathcal{C}_1$ is of the form $\mathcal{C}_1=\oplus_{\alpha\in\mathcal{Z}}L_\alpha,$ for some holomorphic line bundles $L_\alpha.$ Moreover, we can write $\phi\in \oplus_{\alpha\in\mathcal{Z}}L_\alpha$ and decompose accordingly. To see that $\otimes_{\alpha\in\mathcal{Z}}L_\alpha^{n_\alpha}$ is trivial, note that the action of $H$ on $\mathfrak{g}_\alpha$ is by $\exp(\xi)\cdot e_\alpha=\exp(\alpha(\xi))e_\alpha$, for $\xi\in\mathfrak{h}$, so the induced action of $H$ on $\otimes_{\alpha\in\mathcal{Z}}\mathfrak{g}_\alpha^{\otimes n_\alpha}$ is trivial. Thus, any local $H$-trivialization identifies $\otimes_{\alpha\in\mathcal{Z}}L_\alpha^{\otimes n_\alpha}$ with the $H$-trivial line bundle.  
\end{proof}
\begin{remark}\label{rem: no phi_i zero}
It will be useful to know for the proof of Theorem B that if $\phi$ is stable and simple and not a fixed point of the entire $\C^*$-action on the moduli space of $G$-Higgs bundles, then no $\phi_\alpha$ is identically zero. This follows from \cite[Theorem 4.2.2]{C}.

Satisfying results around polystability have recently been established. A cyclic $G$-Higgs bundle can be interpreted as a $(G_0,\g_1)$-Higgs pair (see \cite[\S 5]{GPRi}). Higgs pairs have their own notions of stability and polystability, and a cyclic $G$-Higgs bundle is polystable if and only if the Higgs pair is polystable (see \cite[Proposition 5.7]{GPRi}). In \cite[\S 5]{Mcintosh}, McIntosh shows that if no $\phi_\alpha$ is zero, then the Higgs pair is stable. For $\C^*$-fixed points with $\phi_{-\delta}=0$, in \cite[Proposition 5.4]{Mcintosh}, McIntosh gives a precise characterization for polystability of the Higgs pair in terms of the degrees of the $L_\alpha$'s (which we state in \S \ref{subsec: t to zero}). Polystability for general $\C^*$-fixed points can be determined following the discussion in \cite[\S 4.1]{Mcintosh} (see also \cite[Theorem 4.2.2]{C}).
\end{remark}

\subsection{Hitchin's equations for cyclic $G$-Higgs bundles}\label{sec: equation}
We now prove Theorem A. A number of the results in this subsection hold for more general cyclic $G$-Higgs bundles, with Cartan subgroups replaced with $G_0$'s. We focus on the Coxeter cyclic case to keep things simpler, and because we're mainly concerned with Theorem A. 

\subsubsection{Principal bundles}\label{subsubsec:hitchin-equation-start}  Over the course of the proof of Theorem A, it will be convenient to choose good trivializations for $\textrm{ad} P$, as well as explicit realizations of reductions of the structure group. This will be achieved through Lemma \ref{lm:principal-nonsense} below. 

  Let $P$ be a holomorphic principal $G$-bundle over a complex manifold $X$, and denote by $\mathrm{Isom}(\textrm{ad} P,\mathfrak{g})$ the bundle whose fiber over $x\in X$ is the set of isomorphisms $(\textrm{ad} P)_x\to\mathfrak{g}$. This admits a right $\mathrm{GL}(\mathfrak{g})$ action with quotient $X$, i.e., it is a principal $\mathrm{GL}(\mathfrak{g})$-bundle. 
        \begin{lem}\label{lm:principal-nonsense}
          Assume that $G$ is of adjoint type. Then there is an embedding 
            \begin{align*}
                \Psi_P\colon P\xhookrightarrow{} \mathrm{Isom}(\textrm{ad} P, \mathfrak{g})
           \end{align*}
            that is equivariant with respect to the (faithful) adjoint representation $G\xhookrightarrow{}\mathrm{GL}(\mathfrak{g})$.
        \end{lem}
        Suppose now that we are working over a local patch $U\subset X$ such that $P|_U$ is trivial. Then $P$ admits a holomorphic section $\xi:U\to P$. In particular, $\Psi_P\circ\xi$ defines a holomorphic section of $\mathrm{Isom}(\textrm{ad} P, \mathfrak{g})$ that can be identified with a trivialization $(\textrm{ad} P)|_U\to U\times\mathfrak{g}$. We refer to such a trivialization as a $P$-trivialization.
        \par In the coming proof of Theorem A we will often work in $P$-trivializations. We will also typically construct reductions of the structure group $Q\subset P$ by specifying the set $\Psi_P(Q)\subset\Psi_P(P)\subset\mathrm{Isom}(\textrm{ad} P, \mathfrak{g})$, and then work in $Q$-trivializations of $\textrm{ad} P$. This is equivalent to working in trivializations from $\Psi_P(Q)$. Occasionally, $Q$ will have a structure group that is merely a real torus (rather than a complex Lie group), and in this case, the analogue of Lemma \ref{lm:principal-nonsense} still holds, with the caveat that the embedding $\Psi_P$ is not holomorphic, and hence neither are the $Q$-trivializations.
        \par Finally, we include the construction of $\Psi_P$ from Lemma \ref{lm:principal-nonsense} for completeness. 
        \begin{proof}[Proof of Lemma \ref{lm:principal-nonsense}]
            Identify $\textrm{ad} P$ as the quotient of $P\times\mathfrak{g}$ by the $G$ action given by $g\cdot (p, \xi)=(pg^{-1}, g\xi)$. The desired embedding is given by 
            \begin{align*}
                \Psi_P\colon\;\;\;P&\longrightarrow \mathrm{Isom}(\textrm{ad}P, \mathfrak{g}), \\
                p &\longrightarrow \left([(pg, \xi)]\to g\xi g^{-1}\right).
            \end{align*}
        \end{proof}


    \subsubsection{From Higgs bundles to (\ref{eq: firstthm})}\label{subsubsec:hitchin-equations-middle}
    Here we prove the first part of the equivalence of Theorem A. Let $(P, \phi)$ be a stable cyclic $G$-Higgs bundle on a closed Riemann surface $S$ with holomorphic gauge transformation $s:P\to P$ of order $m$ such that $s^*\phi=\zeta\phi$, where $\zeta$ is an $m^{th}$ root of unity (here, $G$ is not yet assumed to be of adjoint type). The gauge transformation $s$ induces an automorphism of $\textrm{ad} P$, for which $\phi$ lies in the $\zeta$-eigenspace.  Let $\rho:\textrm{ad} P\to\textrm{ad} P$ be the anti-linear involution that solves the Hitchin equation, giving rise to Hermitian metric $h$ with Chern connection $\nabla_h.$ First we prove the following. 
\begin{prop}\label{nablapreserves}
   $s^*\rho = \rho$, and $s^*\nabla_h=\nabla_h$.
\end{prop}
The above is due to Baraglia in the case of the Hitchin section (in the proof of \cite[Proposition 4.1]{Baraglia2010CyclicHB}), and the general case of $s^*\nabla_h=\nabla_h$ is in the proof of \cite[Theorem 4.2.2]{C}. 
\begin{proof}
The anti-linear involution $s^*\rho$ solves the Hitchin equations for $(P,s^*\phi)=(P,\zeta\phi).$ Thus, from Remark \ref{rem: S^1 invariance}, $s^*\rho$ also solves the Hitchin equations for $(P,\phi).$ Since $(P,\phi)$ is stable, the uniqueness theory for Hitchin's equations on closed surfaces applies, and implies that $s^*\rho=\rho$. Moreover, the Chern connection of $s^*\rho$ is $s^*\nabla_h,$ and hence $s^*\nabla_h =\nabla_h.$ 
\end{proof}
Note that $s^*\nabla_h =\nabla_h$ implies that $\nabla_h$ preserves the eigenbundle splitting $\textrm{ad}P=\oplus_{j=0}^{m-1}\mathcal{C}_j$. 
\begin{remark}\label{rem: harmonic lift}
    Since the structure group reduces to $G_0$ and $s^*\rho=\rho,$ the harmonic map lifts to the homogeneous space $G/(K\cap G_0)$. Since $G/(K\cap G_0)$ is reductive, it is pseudo-Riemannian. Collier proves that it is a naturally reductive homogeneous space (see \cite[Definition 4.3.1, Lemma 4.3.3]{C}), which implies this lift is in fact harmonic (see \cite{Wo}). In the Coxeter cyclic case, this harmonic lift is a $\tau$-primitive harmonic map as considered in \cite{Mcintosh}.
\end{remark}


        \par From this point on, we assume that the group $G$ is of adjoint type and that $(P,\phi)$ is Coxeter cyclic, so that $m$ is the Coxeter number $r$. We now think of $P$ as the reduction of the principal bundle $\mathrm{Isom}(\textrm{ad} P, \mathfrak{g})$ as in Lemma \ref{lm:principal-nonsense} and its surrounding paragraphs. Let $g$ be a Coxeter element of $G$, and let $T$ be a maximal torus containing $g$. Let $H$ be the Cartan subgroup that is the complexification of $T$, and let $K$ be the maximal compact subgroup of $G$ that contains $T$. Let $\hat{\rho}:\mathfrak{g}\to\mathfrak{g}$ be the anti-linear involution with fixed point set $\mathfrak{k}$, the Lie algebra of $K$. Let the Lie algebras of $T$ and $H$ be $\mathfrak{t}$ and $\mathfrak{h}$ respectively. We can define $\hat{\rho}$ exactly as in \S 3.1. Note that these choices are also found in Baraglia's paper \cite{Baraglia2010CyclicHB}. By standard Lie theory, any such choice is conjugate to that of Baraglia (for a justification, one can see \cite[Proposition 3.1.2]{L2}). 
        \par We define $P_H\subset P$ to be the $\Psi_P$ pre-image of the set of linear isomorphisms $A:(\textrm{ad} P)_x\to\mathfrak{g}$ in $\mathrm{Isom}(\textrm{ad} P, \mathfrak{g})_x$, for $x\in S$, such that 
        \begin{align*}
            A(s\xi)=\mathrm{Ad}_g A(\xi).
        \end{align*} 
        Let $P_K\subset P$ be the $\Psi_P$ pre-image of the set of isomorphisms $A$ with 
        \begin{align*}
            A({\rho}\xi)=\hat{\rho}(A(\xi)).
        \end{align*}
        Note that $P_K, P_H$ are preserved by the right actions of $K, H$, respectively, and moreover $P_K/K=S$, $P_H/H=S$. Indeed, both quotients are non-empty, since the centralizer of $g$ is $H$ and the normalizer of $K$ in $G$ is $K$, and both quotients are equal to $S$ (rather than a cover) because any two Coxeter elements are conjugate and any two maximal compacts are conjugate. Thus $P_K, P_H$ are principal bundles with structure groups $K, H$. Moreover, $P_H$ is a holomorphic principal bundle, since $s$ and $\Psi_P$ are both holomorphic. This reduction $P_H$ is isomorphic to the previously defined reduction $\mathcal{C}_0$, and henceforth we identify the two reductions. Observe that $P_K$ is precisely the reduction of the structure group of $P$ to the maximal compact $K$ that corresponds to $\rho$, and hence the connection $\nabla_h$ is induced from a connection on $P_K$.
        \par Since $s^*\rho=\rho,$ $P_T=P_K\cap P_H$ is non-empty and defines a principal $T$-bundle. We now study how the metric and connection, $h$ and $\nabla_h$, interact with the reduction $P_T$.
        \begin{lem}\label{claim:conn-induced}
        $\nabla_h$ is induced from a connection on $P_T$.
        \end{lem}
        \begin{proof}
            We work locally over an open set $U\subset S$. Pick a (possibly non-holomorphic) $P_T$-trivialization $\textrm{ad} P|_U\cong U\times\mathfrak{g}$. Let the connection be $d+B$ in these coordinates. Then $B$ is a 1-form taking values in $\mathrm{ad}(\mathfrak{k})\subset\mathrm{End}(\mathfrak{g})$, since $\nabla_h$ is induced from a connection on $P_K$. 
            \par Note that from the fact that $s$ preserves $\nabla_h$, it follows that $g$ preserves $d+B$. Thus $B$ takes values in the set $\mathrm{ad}\{\xi\in\mathfrak{k}:\mathrm{Ad}_g\xi=\xi\}=\mathrm{ad}(\mathfrak{h}\cap\mathfrak{k})=\mathrm{ad}(\mathfrak{t})$. Thus $B$ takes values in $\mathfrak{t}$, and hence $\nabla_h$ is induced from the connection $d+B$ on $P_T|_U$.   
        \end{proof}
   
     Recall from Proposition \ref{prop: one side structure}, that there exist holomorphic line subbundles $L_\alpha\subset \textrm{ad}P$ such that $\phi$ takes values in $\oplus_{\alpha\in\mathcal{Z}}L_\alpha$.  
        \begin{lem}\label{claim:orthogonal-subbundles}
            In the metric induced on $\textrm{ad} P$ by $\rho$, the subbundles $\{L_\alpha:\alpha\in\mathcal{Z}\}$ are mutually orthogonal.
        \end{lem}
        \begin{proof}
            In a local $P_T$-trivialization, the reduction $\rho$ is represented by $\hat{\rho}$, and each subbundle $L_\alpha$ is represented by $\mathfrak{g}_\alpha$. For $\alpha,\beta$ distinct non-zero roots, the root spaces $\g_\alpha$ and $\g_{-\beta}$ are $\nu$-orthogonal, and hence $\nu(\mathfrak{g}_\alpha,-\hat{\rho}(\mathfrak{g}_\beta))=\nu(\mathfrak{g}_\alpha,\mathfrak{g}_{-\beta})=0$. Since any two choices of $\hat{\rho}, g$ are conjugate, the result follows.  
        \end{proof}
        We are now ready to express the Hitchin equations in terms of $L_\alpha$. Let $\mu_\alpha$ be the metric on $L_\alpha$ that is induced by the $\rho$-metric on $\textrm{ad} P$, and let $\phi_\alpha\in H^0(S,\mathcal{K}\otimes L_\alpha)$ be the component of $\phi$ in the direction $L_\alpha$.
        
        Fix a small open patch $U\subset S$ and a $P_T$-trivialization $(\textrm{ad} P)|_U\cong U\times\mathfrak{g}$. We abuse notation slightly and denote the coordinate representations of $\phi,\phi_\alpha$ in this trivialization by $\phi,\phi_\alpha$, so that, in particular, $\phi_\alpha$ is a 1-form taking values in $\mathfrak{g}_\alpha$. Moreover, the Hitchin metric on $\textrm{ad} P$ is represented in coordinates as $\nu(\cdot, -\hat{\rho}(\cdot))$. Using Proposition \ref{prop: one side structure}, $\otimes_{\alpha\in\mathcal{Z}}L_\alpha^{\otimes n_\alpha}$ is the trivial flat Hermitian bundle. We compute the two terms in Hitchin's equations (\ref{SDeqns}) separately.
        \par First, the \textbf{$[\phi,\phi^*]$ term. } For $\alpha\neq \beta$ non-zero distinct roots in $\mathcal{Z}$, Lemma \ref{lem: commutators} shows that
            \begin{align*}
                [\phi_\alpha, \phi_\beta^{*_h}]=[\phi_\alpha, -\hat{\rho}(\phi_\beta)]\in [\mathfrak{g}_\alpha, \mathfrak{g}_{-\beta}]=0.
            \end{align*}
            The diagonal terms $[\phi_\alpha,\phi_\alpha^{*_h}]$ are 
            \begin{align*}
                [\phi_\alpha, \phi_\alpha^{*_h}]=[\phi_\alpha, -\hat{\rho}(\phi_\alpha)]\in[\mathfrak{g}_\alpha,\mathfrak{g}_{-\alpha}]\subset\mathfrak{h}.
            \end{align*}
            Note that by basic properties of the Chevalley basis \cite[\S 8.3, Proposition (c), (g)]{humphreys1972introduction}, we have 
            \begin{align*}
                [e_\alpha, -\hat{\rho}(e_\alpha)]&=[e_\alpha, e_{-\alpha}]=h_\alpha=\nu(e_\alpha, e_{-\alpha})\frac{\nu(\alpha, \alpha)}{2} h_\alpha\\
                &=\nu(e_\alpha, -\hat{\rho}(e_\alpha)) \frac{\nu(\alpha, \alpha)}{2}h_\alpha.
            \end{align*}
            In the local $P_T$-trivialization, we can write  $\phi_\alpha = e_\alpha\otimes f_\alpha,$ where $f_\alpha$ is a holomorphic $1$-form. We find that $$[\phi_\alpha,\phi_\alpha^{*_h}]=f_\alpha\wedge \overline{f_\alpha}\otimes [e_\alpha,-\hat{\rho}(e_\alpha)]=f_\alpha\wedge \overline{f_\alpha}\otimes \nu(e_\alpha, -\hat{\rho}(e_\alpha)) \frac{\nu(\alpha, \alpha)}{2}h_\alpha = \mu_\alpha(\phi_\alpha,\phi_\alpha)\frac{\nu(\alpha, \alpha)}{2}h_\alpha.$$ Above, the $\mu_\alpha(\cdot,\cdot)$ is the extension to $L_\alpha$-valued forms, and $\mu_\alpha(\phi_\alpha,\phi_\alpha)=\mu_\alpha(\phi_\alpha\wedge\overline{\phi}_\alpha),$
            which is the map $L_\alpha\otimes \overline{L}_\alpha\to \C$ associated with $\alpha,$ extended to bundle-valued forms and evaluated on $\phi_\alpha\wedge \overline{\phi}_\alpha$. We prefer the notation $\mu_\alpha(\phi_\alpha\wedge\overline{\phi}_\alpha),$ since it makes clear that the object in question is a $2$-form. Putting everything together, we have
            \begin{align}\label{eq: submain-phiphi*}
                [\phi, \phi^{*_h}]&=\sum_{\alpha\in\mathcal{Z}}\mu_\alpha(\phi_\alpha\wedge \bar{\phi}_\alpha) \frac{\nu(\alpha,\alpha)}{2}h_\alpha.
            \end{align}
            In particular, for any root $\alpha,$
            \begin{align}\label{eq: main-phiphi*}
                \alpha([\phi, \phi^{*_h}])=\sum_{\beta\in\mathcal{Z}} \nu(\alpha, \beta)\mu_\beta(\phi_\beta\wedge \bar{\phi}_\beta).
            \end{align}
        
            \par Second, the \textbf{$F(\nabla_h)$ term. } Note that $\nabla_h$ is the Chern connection for the holomorphic bundle $\textrm{ad} P$. Thus $\nabla_h$ is also the Chern connection for the subbundle $\oplus_{\alpha\in\mathcal{Z}}L_\alpha$. Since the Hitchin metric splits into $\oplus_{\alpha\in\mathcal{Z}}\mu_\alpha$ over $\oplus_{\alpha\in\mathcal{Z}}L_\alpha$, by uniqueness of Chern connections, $F(\nabla_h)|_{L_\alpha}$ acts as multiplication by $F(\mu_\alpha)$. 
            \par In our local $P_T$-trivialization, identify $\textrm{ad} P$ with $\mathfrak{g}$ and $L_\alpha$ with $\mathfrak{g}_\alpha$. $F(\nabla_h)$ is a 2-form taking values in $\mathrm{ad}(\mathfrak{g})\subset\mathrm{End}(\mathfrak{g})$. Let $\xi$ be an $\mathfrak{h}$-valued 2-form such that $\alpha(\xi)=F(\mu_\alpha)$. Then $F(\nabla_h)-\mathrm{ad}\xi$ has pointwise kernel containing $\oplus_{\alpha\in\mathcal{Z}}\mathfrak{g}_\alpha$. Recalling that 
            \begin{equation*}
                \{\xi\in\mathfrak{g}:[\xi,\eta]=0\text{ for all }\eta\in\oplus_{\alpha\in\mathcal{Z}}\mathfrak{g}_\alpha\}=0,
            \end{equation*}
            we see that $F(\nabla_h)=\mathrm{ad}(\xi)$. We deduce that for any root $\alpha,$ 
            \begin{equation}\label{eq: main-fnablah}
                \alpha(F(\nabla_h))=F(\mu_\alpha).
            \end{equation}
        
\begin{proof}[Proof of Theorem A, first side]
We fix a trivialization as above. Since the simple roots span $\mathfrak{h}^*,$ $F(\nabla_h)$ and $[\phi,\phi^*]$ are determined by their values on the simple roots (and certainly, as well, by their values on the extended simple roots). Combining (\ref{eq: main-phiphi*}) with (\ref{eq: main-fnablah}), applying such a root $\alpha\in \mathcal{Z}$ gives $$F(\mu_\alpha) + \sum_{\beta\in\mathcal{Z}} \nu(\alpha, \beta)\mu_\beta(\phi_\beta\wedge \bar{\phi}_\beta)=0.$$ The expression above is independent of the trivialization, and hence we have (\ref{eq: firstthm}).
\end{proof}

\subsubsection{From (\ref{eq: firstthm}) to the Higgs bundle}\label{subsubsec:hitchin-equation-end} 
Finally, we prove the other side of Proposition \ref{prop: structure}, and then the other side of Theorem A.

\begin{proof}[Proof of Proposition \ref{prop: structure}, other side]
Suppose that we are given a collection of line bundles $\{L_\alpha:\alpha\in\mathcal{Z}\}$, sections $\phi_\alpha\in H^0(S,L_\alpha\otimes \mathcal{K})$, and a holomorphic isomorphism 
\begin{align*}
    \Theta: \otimes_{\alpha\in\mathcal{Z}}L_\alpha^{\otimes n_\alpha}\to\mathcal{O}.
\end{align*}
Let $\mathrm{Isom}(\mathfrak{g}_\alpha,L_\alpha)$ be the $H$-bundle over $S$ whose fiber over a point $x\in S$ is the set of isomorphisms from $\g_\alpha\to (L_\alpha)_x$.
Define the space $P_H\subset \prod_{\alpha\in\mathcal{Z}}\mathrm{Isom}(\mathfrak{g}_\alpha,L_\alpha)$ as  
\begin{align*}
P_H=\left\{(A_\alpha:\alpha\in\mathcal{Z}): \Theta\left(\prod_{\alpha\in\mathcal{Z}}A_\alpha(\tilde{e}_\alpha)^{n_\alpha}\right)=1\right\} ,
\end{align*}
where $\tilde{e}_\alpha=\frac{e_\alpha}{\sqrt{\nu(e_\alpha,-\hat{\rho}( e_\alpha))}}$.
Note that $P_H$ admits a natural (right) $H$-action defined by $(A_\alpha:\alpha\in\mathcal{Z})\cdot h=(A_\alpha\circ\mathrm{Ad}_h:\alpha\in\mathcal{Z})$, with $P_H/H=S$, and thus $P_H$ defines a principal $H$-bundle over $S$. Since $\Theta$ is holomorphic, so is $P_H$. Observe that $P_H\times_{\mathrm{Ad}}\mathfrak{g}_\alpha$ admits a natural isomorphism with $L_\alpha$ given by 
\begin{equation*}
    [((A_\beta:\beta\in\mathcal{Z}), e_\alpha)]\longrightarrow A_\alpha(e_\alpha).
\end{equation*}
This implies that $\oplus_{\alpha\in\mathcal{Z}}L_\alpha$ admits a natural embedding into $P_H\times_{\mathrm{Ad}}\mathfrak{g}\cong \textrm{ad} P$, where $P$ is the associated fibre bundle to $P_H$ by the left action of $H$ on $G$. Note that such a $P$ naturally admits a structure of a holomorphic principal $G$-bundle. Thus we obtain a $G$-Higgs bundle $(P,\phi)$ where $\phi=\sum_{\alpha\in\mathcal{Z}}\phi_\alpha$ takes values in $\oplus_{\alpha\in\mathcal{Z}}L_\alpha\subset\textrm{ad} P$. We define the transformation $s:P\to P$ by $s([(q, x)]) = [(q, gxg^{-1})]$, which, since $g$ commutes with elements in $H$, determines a well-defined holomorphic gauge transformation of $P$. Since $\mathrm{Ad}_g$ acts on $\mathfrak{g}_\alpha$ by multiplication by $\zeta$, it follows that $s^*\phi=\zeta \phi$. Thus $(P,\phi)$ is a Coxeter cyclic $G$-Higgs bundle. This completes the proof of Proposition \ref{prop: structure}.
\end{proof}

\begin{proof}[Proof of Theorem A, other side]
Suppose now that $\mu_\alpha$ is a family of metrics having (\ref{eq: firstthm}). We define $P_T\subset P_H$ to be the set of $(A_\alpha:\alpha\in\mathcal{Z})$ such that $A_\alpha(e_\alpha)$ has norm $\nu(e_\alpha,-\hat{\rho}(e_\alpha))=\frac{2}{\nu(\alpha,\alpha)}$ with respect to $\mu_\alpha$. Then $P_T$ is a reduction of the structure group of $P_H$ from $H$ to $T$. Note that the bundle $\oplus_{\alpha\in\mathcal{Z}}L_\alpha$ has a metric $h$ given by the orthogonal sum of the $\mu_\alpha$'s, for which the Chern connection $\nabla_h$ is a sum of the Chern connections of the individual $\mu_\alpha$'s. It is easy to see that $\nabla_h$ is induced from a connection on $P_T$, which we denote $\nabla$. 
\par Working in a local $P_T$-trivialization of $\textrm{ad} P=P_T\times_\mathrm{Ad} \mathfrak{g}$, we define $\rho:\textrm{ad} P\to\textrm{ad} P$ to be equal to $\hat{\rho}$, associated to the Cartan subalgebra $\mathfrak{h}\subset\mathfrak{g}$. By arguments analogous to those of the previous section, we see that (\ref{eq: firstthm}) implies
\begin{equation*}
\oplus_{\alpha\in\mathcal{Z}}\mathfrak{g}_\alpha\subset\ker(F(\nabla) + [\phi, \phi^*]).
\end{equation*}
Thus $F(\nabla)+[\phi,\phi^*]=0$.
\end{proof} 

\begin{remark}\label{rem: not closed}
    The only point where we use that $S$ is closed and that $(P,\phi)$ is stable is for the assertion $s^*\rho=\rho$ in Proposition \ref{nablapreserves}. Thus, if we remove the closedness assumption on $S$ and impose no stability condition on $\phi,$ but assume that $s^*\rho=\rho,$ then the conclusion of Theorem A still holds.
\end{remark}

Theorem A was stated and proved for groups of adjoint type. Without this assumption, similar to Proposition \ref{prop: structure}, the forward direction of Theorem A still holds. Indeed, under the adjoint representation, a Coxeter cyclic G-Higgs bundle $(P,\phi)$ produces a Coxeter cyclic $\textrm{ad}G$-Higgs bundle $(P_{\textrm{ad}G},\phi)$. Furthermore, if $s$ makes $(P,\phi)$ Coxeter cyclic, a solution to the self-duality equations for $(P,\phi)$ with involution $\rho$ such that $s^*\rho = \rho$ induces a solution to the corresponding equations for $(P_{\textrm{ad}G},\phi)$ with the same property. Note that the adjoint representation induces an isomorphism from $\textrm{ad}P\to \textrm{ad}P_{\textrm{ad}G}$, so $\phi$ really is the Higgs field on $P_{\textrm{ad}G}$. The line bundles $L_\alpha$ and sections $\phi_\alpha$ associated with $(P,\phi)$ and $(P_{\textrm{ad}G},\phi)$ are equal, and from Theorem A, we get metrics $\mu_\alpha$ on $L_\alpha$ solving (\ref{eq: firstthm}).

\begin{remark}\label{rem: constructing gauge transformations}
    In the proof of Proposition \ref{prop: structure}, we used the Coxeter element $g$ to induce a holomorphic gauge transformation of the $G$-bundle $P$. More generally, keeping in the setting of the proof, out of certain collections of isomorphisms $\g_\alpha \to \g_\alpha$, $\alpha\in \mathcal{Z}$, we can build holomorphic gauge transformations of $P.$ We make use of this construction in Section \ref{subsec: t to zero}. For $\alpha\in \mathcal{Z}$, let $s_\alpha^{t_\alpha}: \g_\alpha\to \g_\alpha$ be multiplication by $t_\alpha \in \C^*$. 
    Each $s_\alpha^{t_\alpha}$ induces an automorphism of $\textrm{Isom} (\g_\alpha, L_\alpha)$ via $\varphi_\alpha \mapsto \varphi_\alpha \circ s_\alpha^{t_\alpha}.$ Taking them all together,
    $\oplus_{\alpha\in Z} s_\alpha^{t_\alpha}$ induces an automorphism of $\prod_{\alpha\in\mathcal{Z}} \textrm{Isom} (\g_\alpha, L_\alpha)$ that commutes with the action of $H$, i.e., a gauge transformation. If $\prod_{\alpha\in \mathcal{Z}} t_\alpha^{n_\alpha} =1,$ then the gauge transformation preserves $P_H$. Extending the structure group to $P$, the gauge transformation comes along for the ride and determines a holomorphic gauge transformation $s$ of $P$ such that $s^*\phi_\alpha$ is $t^\alpha \phi_\alpha.$

   The backward direction of Proposition \ref{prop: structure} and the above construction are specialized to a group $G$ of adjoint type. To understand the general case, let $G'$ be another Lie group, which we assume admits a surjective homomorphism $G'\to G$ with discrete kernel. Assume further that we are given a $G'$-Higgs bundle $(P',\phi)$ that induces $(P,\phi)$ through this homomorphism. Then $s$ lifts to $P'$: any gauge transformation of $(P,\phi)$ is equivalent to a section of $P\times_{\textrm{Ad}} G$, and the section $s$ is equivalent to a constant map to $T$ that's been factored through the inclusions $T\to P_T \times_{\textrm{Ad}} T \to P\times_{\textrm{Ad}} G$. Clearly, we can choose a lift of this constant map to $G$ (lifts are indexed by the kernel of $G'\to G$), which then induces a lift of $s$ to $P'$. Of course, the behaviour on the $\phi_\alpha$'s is the same.
\end{remark}

\subsection{Cyclic $G^{\R}$-Higgs bundles}\label{sec: G^R higgs}
With the proof of Theorem A complete, we digress a bit in order to discuss the extra symmetries that come into play when working with $G^{\R}$-Higgs bundles, where $G^{\R}$ is the split real form of the complex simple Lie group $G$. This is the setting for cases $\textrm{A}_n$ and $\textrm{E}_6$ in Theorem C'. 
The definition for cyclic $G^{\R}$-Higgs bundles is more or less the same as in the case of complex groups; before the definition, we just point out that any gauge transformation $s$ of a principal bundle induces an automorphism of every associated bundle. For $K$ the maximal compact subgroup of $G^{\R},$ with complexification $K^{\C},$ we denote principal $K^{\C}$-bundles by $P_{K^{\C}}.$  
\begin{defn}
   A $G^{\R}$-Higgs bundle $(P_{K^{\C}},\phi)$ is cyclic if there exists a holomorphic gauge transformation $s$ of $P_{K^{\C}}$ of finite order $m$ such that $s^*\phi=e^{\frac{2\pi i}{m}}\phi$. If $m$ is the Coxeter number of $G,$ we say that it is Coxeter cyclic.
\end{defn}

Write the Lie algebras of $G^{\R}$ and $K$ as $\g^{\R}$ and $\mathfrak{k}$ respectively, and let $\mathfrak{p}$ be the Killing orthogonal complement of $\mathfrak{k}$. As usual, $\g$ is the Lie algebra of $G$. It follows from Hitchin's construction in \cite[\S 6]{Hi} that we can find an involution $\sigma_0: G^{\R}\to G^{\R}$ that splits the Lie algebra $\g^{\R}$ as $\g^{\R}=\mathfrak{k}\oplus \mathfrak{p}$ and such that, if we complexify to an involution of $G$ (which we continue to denote by $\sigma_0)$, then, on the Lie algebra level, $\sigma_0$ preserves a Cartan subalgebra $\mathfrak{h}$ of $\g$ and a choice of positive roots $\Delta^+$ (it must then also preserves the simple roots $\Pi$). From \cite{Hi}, $\sigma_0=\textrm{ad}_{g_\pi}\circ w,$ where $g_\pi$ is a group element chosen so that $\textrm{ad}_{g_\pi}$ acts by rotation by $\pi$ in a principal $\mathfrak{sl}(2,\C)$, and $w$ is an automorphism in the component labelled by some $\hat{w}\in \textrm{Aut}(\Pi)$. The element $\hat{w}$ is trivial unless $G$ is of type $\textrm{A}_n, \textrm{D}_{2n+1},$ or $\textrm{E}_6,$ in which case $\hat{w}$ has order $2$. Considering the action on the simple roots, it acts by the element $\hat{w}$. Writing $\sigma_0$ as well for the action on simple roots, it is easily checked that $\sigma_0$ takes the root space $\g_\alpha \subset \g$ to $\g_{\sigma_0(\alpha)}.$ 

The involution $\sigma_0$ extends to $P_{K^{\C}}\times G$ via $\sigma_0\cdot (p,g) = (p,\sigma_0(g))$ and then descends to the extension of structure group $P=P_{K^{\C}}\times_{K^{\C}} G.$ Of course, $\sigma^* \phi = -\phi.$ As well, by descending the operation $s\cdot (p,g) = (s(p),g),$ $s$ induces a gauge transformation of $P$, which we still denote by $s$, which makes the $G$-Higgs bundle $(P,\phi)$ Coxeter cyclic. By construction, the two operations commute, i.e.,  $\sigma^* s = s$. In the language of \cite{GP}, \cite{GPRi}, for these groups, $(P,\phi)$ is also $\theta$-cyclic for some $\theta\in \textrm{Aut}(G)$ of order $2$.

Now, we explain the analogue of Proposition 1.2 for cyclic $G^{\R}$-Higgs bundles. 
We resume the notation from \S \ref{sec: cyclic definition} and \ref{sec: equation}. Relating to our framework, since $\sigma^* s = s$, $\sigma$ necessarily preserves the $s$-eigenbundle splitting of $\textrm{ad}P$, $\textrm{ad}P=\oplus_{j=0}^{m-1} \mathcal{C}_j$. We impose that $(P_{K^{\C}},\phi)$ is Coxeter cyclic. Then, $\mathcal{C}_1$ splits into line bundles, and after identifying our index set for these line bundles with the set of simple roots $\Pi$ preserved by $\sigma_0,$ we can write $\mathcal{C}_1=\oplus_{\alpha\in\mathcal{Z}}L_\alpha$, and we know that $\sigma$ identifies $L_\alpha$ with $L_{\sigma_0(\alpha)}.$

\begin{prop}
     Let $G$ be of adjoint type. A Coxeter cyclic $G^{\R}$-Higgs bundle over a Riemann surface $S$ is equivalent to a collection of line bundles $\{L_\alpha:\alpha\in\mathcal{Z}\}$ with an isomorphism $\Theta: \bigotimes_{\alpha\in\mathcal{Z}} L_\alpha^{\otimes n_\alpha}\to \mathcal{O}$, holomorphic sections $\phi_\alpha\in H^0(S,L_\alpha\otimes \mathcal{K})$, and a collection of isomorphisms $\sigma_\alpha: L_\alpha\to L_{\sigma_0(\alpha)}$ such that $\sigma_\alpha(\phi_\alpha) = -\phi_{\sigma_0(\alpha)}.$ 
\end{prop}
\begin{proof}
Most of the first direction was explained above. We label the maps from $L_\alpha\to L_{\sigma_0(\alpha)}$ by $\sigma_\alpha,$ and the relation $\sigma^*\phi=-\phi$ translates to $\sigma_\alpha(\phi_\alpha) = - \phi_{\sigma_0(\alpha)}.$ Note that, as in Proposition \ref{prop: one side structure}, this direction does not require $G^{\R}$ to be of adjoint type.

For the converse direction, we step back into the proof of Proposition \ref{prop: structure}. The extra data compared to Proposition \ref{prop: structure} is that of the isomorphisms $\{\sigma_\alpha\}_{\alpha\in \mathcal{Z}}$, so we only need to show that these isomorphisms give a way to reduce the $G$-Higgs bundle $(P,\phi)$ obtained through Proposition \ref{prop: structure} to a $G^{\R}$-Higgs bundle. Toward this, observe that the $\sigma_\alpha$'s define isomorphisms $\hat{\sigma}_\alpha: \textrm{Isom}(\g_\alpha,L_\alpha) \to \textrm{Isom}(\g_{\sigma_0(\alpha)},L_{\sigma_0(\alpha)})$ via $\varphi_\alpha\mapsto \sigma_\alpha\circ \varphi_\alpha \circ \sigma_0^{-1}$. The sum of the $\hat{\sigma}_\alpha$'s then gives a gauge transformation of the bundle $P_H$, which induces a further gauge transformation $\sigma$ on the extension of structure group $P$. It is clear that $\sigma^*\phi=-\phi,$ and moreover that $\sigma$ reduces $(P,\phi)$ to a $G^{\R}$-Higgs bundle. 
\end{proof}
To see how Theorem A is refined for $G^{\R}$-Higgs bundles, first observe that the involution $\rho:\textrm{ad}P\to \textrm{ad}P$ solving Hitchin's equations satisfies: $\sigma^*\rho=\rho.$ The well-known line of reasoning is similar to that of Proposition \ref{nablapreserves}: $\sigma^*\rho$ solves Hitchin's equations for $(P,\sigma^*\phi)=(P,-\phi),$ and we then apply Remark \ref{rem: S^1 invariance}. Recall that $\rho$ induces Hermitian metrics $\mu_\alpha$ on $L_\alpha$, and that the equation from Theorem A concerns these metrics. It follows from $\sigma^*\rho=\rho$ that the restricted map $\sigma|_{L_\alpha}=\sigma_\alpha:L_\alpha\to L_{\sigma_0(\alpha)}$ identifies the Hermitian metric $\mu_\alpha$ on $L_\alpha$ with the metric $\mu_{\sigma_0(\alpha)}$ on $L_{\sigma_0(\alpha)}$. Moreover, $\mu_\alpha(\phi_\alpha\wedge \phi_\alpha) = \mu_{\sigma_0(\alpha)}(\phi_{\sigma_0(\alpha)}\wedge \phi_{\sigma_0(\alpha)}).$ Consequently, when the action of $\sigma_0$ on the Dynkin diagram is non-trivial, Hitchin's equations reduce to a system with fewer variables. We will explain some special cases in detail in  \S \ref{sec: thms C and C'}.
\begin{remark}
    Olive and Turok originally pointed out in \cite{OLIVE1983470} that one can obtain a reduced version of the Toda lattice equations when the Dynkin diagram has extra symmetry. In the Higgs bundle context, Baraglia observed the reduction for Coxeter cyclic Higgs bundles in the Hitchin component \cite[\S 2.2.4]{Baraglia2010CyclicHB}. Similar to (\ref{eq: firstthm}), the general equations that we find here and in \S \ref{sec: thms C and C'} are equivalent to certain equations in \cite{Mcintosh} (see, in particular, \cite[\S 4.2]{Mcintosh}).
\end{remark}

\subsection{The Hitchin section}\label{sec: hitchin section}
Let $G$ be of adjoint type. Let $p_1,\dots, p_l$ be a basis for the algebra of $G$-invariant polynomials on $\g$, ordered so that if $m_i=\deg p_i,$ then $m_i\leq m_{i+1}.$ Applying $p_i$ to a Higgs field returns a holomorphic $m_i$-differential. The Hitchin fibration relative to this basis is $$\mathcal{M}_S(G)\to \oplus_{i=1}^l H^0(S,\mathcal{K}^{m_i}), \hspace{2mm}[(P,\phi)]\mapsto (p_1(\phi),\dots, p_l(\phi)).$$ 
Following Hitchin in \cite{Hi} and Baraglia in \cite{Baraglia2010CyclicHB}, we outline the construction of Hitchin's section $$b: \oplus_{i=1}^l H^0(S,\mathcal{K}^{m_i})\to \mathcal{M}_S(G)$$ and describe the Coxeter cyclic $G$-Higgs bundles in the Hitchin component. By the third description of the Coxeter number in \S \ref{sec: aut}, the Coxeter number is $r=m_l$.

\begin{remark}
    If we use the Hitchin fibration on a $G$-Higgs bundle $(P,\zeta_k\phi)$, we obtain $$(\zeta_k^{m_1}p_1(\phi),\zeta_k^{m_2}p_2(\phi),\dots, \zeta_k^{m_l}p_l(\phi)).$$  It follows that if $(P,\phi)$ is Coxeter cyclic, then $p_1(\phi)=\dots=p_{l-1}(\phi)=0,$ and so the Hitchin fibration returns just  an $r$-differential ($r=m_l$). 
\end{remark}

The $G$-Higgs bundles in Hitchin's section are all made over a common principal $G$-bundle $P$. To define $P$, start by choosing a square root $\mathcal{K}^{1/2}$ of $\mathcal{K}$ and define $P_{2}$ to be the 
frame bundle of $\mathcal{K}^{1/2}\oplus \mathcal{K}^{-1/2},$ which is a principal $\textrm{SL}(2,\C)$-bundle. We obtain a $G$-bundle using a suitable map $i_G: \textrm{SL}(2,\C)\to G$. To construct such a map $i_G$, as in \S \ref{sec: root systems}, let $\mathfrak{h}\subset \mathfrak{g}$ be a Cartan subalgebra and fix a set of positive roots $\Delta^+$ with simple root set $\Pi$ and a Chevalley base $\{h_\beta, e_\alpha, e_{-\alpha}: \beta\in \Pi, \alpha\in \Delta^+\}.$  Let $$x=\frac{1}{2}\sum_{\beta\in\Delta^+}h_\beta.$$ There are $r_\alpha\in\R$ such that $x=\sum_{\alpha\in \Pi} r_\alpha h_\alpha.$  Set $e=\sum_{\alpha\in \Pi}\sqrt{r_\alpha}e_{\alpha}$ and $\tilde{e}=\sum_{\alpha\in \Pi}\sqrt{r_\alpha}e_{-\alpha}.$ The element $e$ is a regular nilpotent, and the subalgebra $\mathfrak{s}\subset \p$ generated by $\{x,e,\tilde{e}\}$ is isomorphic to $\mathfrak{sl}(2,\C)$ and called a principal $\mathfrak{sl}(2,\C).$ Our map $i_G: \textrm{SL}(2,\C)\to G$ is obtained by exponentiation. Equipped with $i_G$, we define the principal $G$-bundle $P$ by $P=P_{2}\times_{i_G} G$.

To see the adjoint bundle, observe that if $y$ is in the root space of $\alpha\in \Delta$, then $[x,y]$ is the height of $\alpha$. The decomposition of $\g$ under the adjoint action of $x$, $$\g=\bigoplus_{m=-M}^M \g_m, \hspace{1mm} \g_m=\{Y\in \g^{\C}: [x,Y]=mY\},$$ is called the height decomposition (note $M=r-1$). We point out that each $\g_m$ decomposes as the sum of the root spaces of height $m$. It is explained in \cite{Hi} that the adjoint bundle of $P$ identifies with $$\bigoplus_{m=-M}^M\g_m\otimes\mathcal{K}^m,$$ with standard holomorphic structure induced from $\mathcal{K}$. The holomorphic class of $P$ and its adjoint bundle are in fact independent of the original choice of $\mathcal{K}^{1/2}$. 

The centralizer of $e$ is spanned by elements $f_1,\dots, f_l,$ which can be taken as highest weight vectors for the irreducible components of the adjoint action of $\mathfrak{s}$ on $\mathfrak{g}.$ These elements are used in defining Hitchin's section.
   
\begin{defn}
    The Hitchin section $b$ assigns a tuple $(q_1,\dots, q_l)$ to the class of the $G$-Higgs bundle $$(P,\phi=\tilde{e}+q_1f_1+\dots +q_lf_l).$$ 
\end{defn}
To interpret $\phi$ as a holomorphic section of $\textrm{ad}P\otimes\mathcal{K},$ we view $\tilde{e}$ as a section of $(\g_{-1}\otimes \mathcal{K}^{-1})\otimes \mathcal{K}$, and each $q_if_i$ as a section of $(\g_{m_i-1}\otimes \mathcal{K}^{m_i-1})\otimes \mathcal{K}.$ It is proved in \cite[\S 5]{Hi} that the $G$-Higgs bundle $(P,\phi)$ is stable. It is also shown in \cite[\S 6-7]{Hi} that $(P,\phi)$ reduces to a $G^{\R}$-Higgs bundle.

Since $b$ is a section, we already know that for every $q=q_l$, $(P,\phi_q:=\tilde{e}+qf_l)$ is Coxeter cyclic (and, as a $G^{\R}$-Higgs bundle); we will produce the explicit gauge transformation. Over a pair of contractible and trivialized open subsets of $\mathcal{K}$, $U_\alpha\times \mathcal{K}$ and $U_\beta\times \mathcal{K}$, let $e^{f_{\alpha\beta}}$ be the transition function for $\mathcal{K}$. The transition functions for $\textrm{ad}P$ are then $\textrm{ad}_{\textrm{exp}(f_{\alpha\beta} x)}$, which preserve the Coxeter automorphism of $\g$ determined by $g=\textrm{exp}(\frac{2\pi i x}{r})$. Thus, $g$ defines a global gauge transformation, and the Higgs bundle $(P,\phi_q)$ is cyclic via $g,$ as can be observed by the formula $[x,\phi_q]=-\tilde{e}+(r-1)qf_l$. Going through the construction of the reduction to $K^{\C}$ from \cite{Hi} associated with the $G^{\R}$-Higgs bundle structure on $(P,\phi)$, one can check directly that $g$ induces a gauge tranformation on the reduced $K^{\C}$-bundle.

Split $\g_{-1}$ into root spaces as $\g_{-1}=\oplus_{\alpha\in\mathcal{Z}} \g_{-\alpha}$. Interpreting  $\tilde{e}$ as in $\g_{-1}$, the projection to $\g_{-\alpha}$ is non-zero for every $\alpha\neq -\delta$. For $(P,\phi_q),$ for $\alpha\in \Pi,$ the associated line bundle $L_\alpha$ identifies with $\g_{-\alpha}\otimes \mathcal{K}^{-1}$, and $L_{-\delta}$ with $\g_{\delta}\otimes \mathcal{K}^{r-1}$ (note $\g_{\delta}=\g_{r-1}$).

Finally, while we worked here with groups of adjoint type, as Hitchin did in \cite{Hi}, we mentioned earlier that, for $G$ of rank at least $2$, Hitchin representations lift to the universal cover of $G$. In the sequel, when referring to a Hitchin representation or a $G$-Higgs bundle in the Hitchin section, we allow for non-adjoint groups $G$, with an object being Hitchin if it comes from a lift of a Hitchin object.

\section{Elliptic systems}
Here we use Theorem A to derive the Bochner-Toda system, which is the main system governing the analytic quantities involved in Theorems B, C, and C'. Then we recall Dai-Li's maximum principle \cite[Lemma 3.1]{Dai2018}.

\subsection{Bochner--Toda system}\label{subsubsec:bochner-toda} 
Fix an ambient metric $\mu$ on the underlying Riemann surface $S$ that's compatible with the Riemann surface structure. Denote by $\Omega_\mu$ the area form associated with $\mu$. If $\mu$ is given in local coordinates by $\mu(z)\abs{dz}^2$, then $\Omega_\mu=\mu(z)\frac{i}{2}dz\wedge d\bar{z}$, and we will occasionally use these expressions. 
\par We now let $(P,\phi)$ be a stable $G$-Higgs bundle on a $S$ with harmonic metric $h$ and equivariant harmonic map $f_\phi$. We will abuse notation slightly and write $\phi$ for $\textrm{ad}(\phi)$ when it is clear that we're working in the adjoint representation. We write $|\phi|_{h,\mu}$ for the norm of $\phi$ with respect to the metrics $h$ on $\textrm{ad} P$ and $\mu$ on $S$. Formally, if $\phi=\phi_0dz,$ for $\phi_0$ a section of $\textrm{ad} P$, then $$|\phi|_{h,\mu}^2=\nu(\phi_0,\phi_0^{*_h})|dz|_\mu^2.$$
Recall the definition of energy density from the introduction: $$e(f_\phi) = \textrm{tr}_\mu f_\phi^*\nu.$$ By equivariance of $f_\phi$, the energy density is invariant under the $\pi_1(S)$ action and hence descends to a function on $S$.

By Proposition \ref{prop: one side structure}, we obtain line bundles $L_\alpha$ and sections $\phi_\alpha\in H^0(S,L_\alpha\otimes \mathcal{K}).$ By Theorem A and the discussion after the proof, we have a collection of metrics $\mu_\alpha$ on $L_\alpha$ solving (\ref{eq: firstthm}). Set $e_\alpha=\Omega_\mu^{-1}\frac{i}{2}\mu_\alpha(\phi_\alpha\wedge\bar{\phi}_\alpha)$.
\begin{prop}\label{prop: energy}
    When $(P, \phi)$ is Coxeter cyclic,
    \[e(f_\phi) = |\phi|_{h,\mu}^2=\sum_{\alpha\in\mathcal{Z}} e_\alpha.\]
\end{prop} 
\begin{proof}
    The first equality is well known and holds for arbitrary (stable) $G$-Higgs bundles. Using the isomorphism induced by the Maurer-Cartan form, $f^*\nu = \nu(\phi,\phi^{*_h}) + \nu(\phi,\phi) + \nu(\phi^{*_h},\phi^{*_h}).$ Taking the trace isolates the first term, and we get $e(f)=\mu^{-1}\nu(\phi,\phi^{*_h})=|\phi|_{h,\mu}^2.$
    \par The quantity $|\phi|_{h,\mu}^2$ agrees with the corresponding quantity for the Higgs field on the associated $\textrm{ad}G$-Higgs bundle (geometrically, we have a local isometry of symmetric spaces that intertwines the harmonic maps). The second equality then follows immediately from Lemma \ref{claim:orthogonal-subbundles}, and arguments analogous to the those in the computation of $[\phi,\phi^{*_h}]$ in \S \ref{subsubsec:hitchin-equations-middle}.
\end{proof}
It makes sense to refer to $e_\alpha$ as the energy of the root $\alpha$. We write down a system for the $e_\alpha$'s, $\alpha\in\mathcal{Z}$. To do so, we additionally need the following local computation. 
\begin{lem}
    Away from the zeros of $e_\alpha$, we have $\Delta_\mu \log e_\alpha=-2i\Omega_\mu^{-1}F(\mu_\alpha)+2K_\mu$, where $K_\mu$ is the sectional curvature of the metric $\mu$. 
\end{lem}
\begin{proof}
Write $\mu=\mu(z)\abs{dz}^2$ in a local holomorphic coordinate $z$ on $S$, and write $\phi_\alpha= fdz$, choosing in the process a holomorphic trivialization of $L_\alpha$. Denote analogously by $\mu_\alpha(z)$ the expression of $\mu_\alpha$ in this trivialization. Then $e_\alpha(z)=\abs{f}^2 \mu(z)^{-1}\mu_\alpha(z)$, and 
\begin{align*}
    \Delta_\mu \log e_\alpha&=\frac{4}{\mu(z)} \frac{\partial^2 \log e_\alpha(z)}{\partial z\partial \bar{z}}=\frac{2i \partial\bar{\partial}\log \mu_\alpha(z)}{\Omega_\mu}-\frac{4}{\mu(z)}\frac{\partial^2\log\mu(z)}{\partial z\partial \bar{z}}\\
    &=-2i \Omega_\mu^{-1}F(\mu_\alpha)+2K_\mu,
\end{align*}
where we used \cite[\S 7, pp 9]{Schoen1997} in the last line.
\end{proof}
From (\ref{eq: firstthm}), we now immediately get the Bochner-Toda system, whose name is inspired by the classical theory of harmonic maps.
\begin{thm}[Bochner--Toda equations]\label{thm: BT eqns}
    Suppose we are in the setting of Theorem A, and that we have a metric $\mu$ on $S$ with sectional curvature $K_\mu$. Then for $e_\alpha=\Omega_\mu^{-1}\frac{i}{2}\mu_\alpha(\phi_\alpha\wedge\bar{\phi}_\alpha)$, away from its zeros we have 
    \begin{equation}\label{eq:bochner-toda}
    \Delta_\mu \log e_\alpha=2K_\mu+4\sum_{\beta\in\mathcal{Z}}\nu(\alpha, \beta) e_\beta,
    \end{equation}
    for all $\alpha\in\mathcal{Z}$. Moreover, using the isomorphism $\otimes_{\alpha} L_\alpha^{n_\alpha}\cong\mathcal{O}$, we have 
    \begin{equation}\label{eq: energy-product}
        \prod_{\alpha\in\mathcal{Z}}e_\alpha^{n_\alpha}=\abs{\prod_{\alpha\in\mathcal{Z}} \left(\mu^{-1/2}\phi_\alpha\right)^{n_\alpha}}^2.
    \end{equation}
\end{thm}
By the definition of $e_\alpha,$ solving (\ref{eq:bochner-toda}) is equivalent to solving Hitchin's equations.

We obtain a variant of the Bochner-Toda equations by setting $\tilde{e}_\alpha = e_\alpha\nu(\alpha,\alpha).$ The system of (\ref{eq:bochner-toda}) becomes 
\begin{equation}\label{eq:bochner-toda-variant}
    \Delta_\mu \log \tilde{e}_\alpha=2K_\mu+2\sum_{\beta\in\mathcal{Z}}a_{\alpha\beta} \tilde{e}_\beta,
    \end{equation}
    where, as in \S \ref{sec: root systems}, $a_{\alpha\beta}=2\frac{\nu(\alpha,\beta)}{\nu(\beta,\beta)}$. 
    The system (\ref{eq:bochner-toda}) will be more useful for proving Theorem B, while (\ref{eq:bochner-toda-variant}) will be for Theorems C and C'. The first advantage of (\ref{eq:bochner-toda-variant}) is that it is invariant under the choice of normalization of the Killing form. Perhaps more significantly, as we'll exploit in \S \ref{sec: curvature}, when $(P,\phi)$ is induced by a $G^{\R}$-Higgs bundle, $\tilde{e}_\alpha=\tilde{e}_{\sigma_0(\alpha)}$, where $\sigma_0$ is the involution from \S \ref{sec: G^R higgs} (this only has an effect when the action on the simple roots is non-trivial). When the context is clear, we omit the word ``variant" and just call (\ref{eq:bochner-toda-variant}) a Bochner-Toda system.

\subsection{Dai-Li Maximum Principle}
Dai-Li's maximum principle is below. We do not apply it to Bochner-Toda systems explicitly, but we apply it to systems that we further derive from Bochner-Toda systems.
\begin{thm}[Lemma 3.1 in \cite{Dai2018}]\label{thm: DL max}
      Let $X$ be a closed Riemannian manifold, and for $i=1,\dots, n,$ let $P_i$ be an isolated subset of $X$ (which can be empty), and $P=\cup_{i=1}^n P_i.$ Let $c_{ij}:X\backslash P\to\mathbb{R}$ be smooth and bounded functions indexed by $1\leq i,j\leq n$, such that 
    \begin{enumerate}
        \item(cooperative) $c_{ii}\geq 0$ and $c_{ij}\leq 0$ for $i\neq j$, 
        \item(column-diagonally dominant) $\sum_{i=1}^n c_{ij}\geq 0$ for all $1\leq j\leq n$, and 
        \item(fully coupled) there is no partition $\{1,2,...,n\}=A\cup B$ such that $c_{ij}=0$ for $i\in A, j\in B$.  
    \end{enumerate}
    Suppose that $u_i:X\backslash P_i\to\mathbb{R}$, $i=1,\dots, n$, are smooth functions that approach $+\infty$ around $P_i$ and solve the system 
    \begin{align*}
        \Delta u_i=\sum_{j=1}^n c_{ij}u_j+\langle Y,\nabla u_i\rangle+f_i\text{ for }i=1,2,...,n,
    \end{align*}
    for some continuous vector field $Y$ and non-positive continuous functions $f_i.$
Consider the following conditions.
\begin{enumerate}
    \item $(f_1,\dots, f_n)\neq (0,\dots, 0).$
    \item $P$ is non-empty.
    \item $\sum_{i=1}^n u_i\geq 0.$
\end{enumerate}
Either condition (1) or (2) implies $u_i>0$ for all $i,$ and condition (3) implies either $u_i>0$ for all $i$ or $u_i\equiv 0$ for all $i.$ 
\end{thm}
We use condition (3) to prove Theorem B, and condition (2) to prove Theorems C and C'. Theorem \ref{thm:dai-li-gen} in the appendix relaxes condition (3).

\section{Monotonicity Theorem}
In this section we prove Theorem B and Corollary B. We then discuss taking the $t$ parameter to zero in \S \ref{subsec: t to zero}. In this section, $G$ is a simple complex Lie group.
\subsection{Proof of Theorem B}
We recall the setup from the introduction. Let $\mu$ be a conformal metric on the closed Riemann surface $S$ and let $(P,\phi)$ be a stable and simple Coxeter cyclic $G$-Higgs bundle on $S$ that is not a fixed point of the $\C^*$-action. For each $t\in \C^*,$ let $f_t:\tilde{S}\to G/K$ be the equivariant harmonic map associated with $(P,t\phi).$ Let $e_\alpha(t)$ be the energy of the root $\alpha$ for the bundle $(P, t\phi)$. We point out that, from Remark \ref{rem: S^1 invariance}, $e_\alpha(t)=e_{\alpha}(|t|).$
\par By Proposition \ref{prop: energy},
\begin{equation}\label{eq: energy formula}
    \mathrm{e}(f_t)=\sum_{\alpha\in\mathcal{Z}}e_\alpha(t),
\end{equation}
and hence Theorem B follows immediately from Lemma \ref{lm:ineq-alpha} below.
\begin{lem}\label{lm:ineq-alpha}
    For all $\alpha\in\mathcal{Z}$, away from the zero set of $\phi_\alpha,$ $e_\alpha(t)$ is strictly increasing in $|t|$.
\end{lem}
\begin{proof}
For any $t$ and $s$, $\log\frac{e_\alpha(t)}{e_\alpha(s)}$ is a smooth function on $S,$ since the divisor of $e_\alpha(t)$, being the divisor of $\phi_\alpha$, does not depend on $t$, and each $\phi_\alpha$ is holomorphic. For $|t|>|s|$, by (\ref{eq:bochner-toda}) we have 
\begin{align*}
    \Delta_\mu \left(n_\alpha\log\frac{e_\alpha(t)}{e_\alpha(s)}\right)=4\sum_{\beta\in\mathcal{Z}} \nu(n_\alpha\alpha, n_\beta^{-1}\beta) \frac{e_\beta(t)-e_\beta(s)}{\log e_\beta(t)-\log e_\beta(s)}\left(n_\beta \log\frac{e_\beta(t)}{e_\beta(s)}\right).
\end{align*}
    We now apply Theorem \ref{thm: DL max} to this system for the functions $(n_\alpha\log\frac{e_\alpha(t)}{e_\alpha(s)}: \alpha\in\mathcal{Z})$.
    \begin{enumerate}
        \item The system is cooperative since by Proposition \ref{prop: angle relations} we have $\nu(\alpha,\alpha)>0$ and $\nu(\alpha,\beta)\leq 0$ for $\alpha\neq\beta$. 
        \item We have 
        \begin{align*}
            \sum_{\alpha\in\mathcal{Z}} \nu(n_\alpha\alpha, n_\beta^{-1}\beta) \frac{e_\beta(t)-e_\beta(s)}{\log e_\beta(t)-\log e_\beta(s)}&=\frac{e_\beta(t)-e_\beta(s)}{\log e_\beta(t)-\log e_\beta(s)}\nu\left(n_\beta^{-1}\beta, \sum_{\alpha\in\mathcal{Z}} n_\alpha\alpha\right)\\&=0,
        \end{align*}
        for all $\beta\in\mathcal{Z}$, so the system is also column diagonally dominant.
        \item Since $\phi$ is stable and simple, Remark \ref{rem: no phi_i zero} guarantees that for all $\alpha\in\mathcal{Z},$ $\phi_\alpha$ is not identically zero. The fact that the system is fully coupled then follows from this assertion together with the fact that the affine Dynkin diagram associated with a simple complex Lie group is connected.
    \end{enumerate}
 Finally, by (\ref{eq: energy-product}),
    \begin{align*}
    \sum_{\alpha\in\mathcal{Z}} n_\alpha\log\frac{e_\alpha(t)}{e_\alpha(s)}=2\log\frac{|t|}{|s|}>0,
    \end{align*}
    so by Theorem \ref{thm: DL max} we have $n_\alpha\log\frac{e_\alpha(t)}{e_\alpha(s)}>0$ for all $\alpha\in\mathcal{Z}$.
    \end{proof}

    \begin{proof}[Proof of Theorem B]
        Apply the formula (\ref{eq: energy formula}) and Lemma \ref{lm:ineq-alpha}.
    \end{proof}

\subsection{Volume entropy} We now prove Corollary B. All that we will use about volume entropy is that it satisfies the following two properties, which are easy to see from the definition.
\begin{enumerate}
\item if $\mu_1\leq \mu_2$, then $\textrm{Ent}(\mu_2)\leq \textrm{Ent}(\mu_1).$
\item We have the scaling property $\textrm{Ent}(t\mu)=t^{-1}\textrm{Ent}(\mu).$
\end{enumerate}
Before getting to the corollary, we prove a general result about volume entropy for stable $G$-Higgs bundles, making use of an object studied in \cite{SS}.  A $G$-Higgs bundle $(P,\phi)$ is called nilpotent if $\textrm{ad}(\phi)$ defines a nilpotent endomorphism of $\g$ in every fiber.
\begin{prop}\label{prop: entropy}
     Let $(P,t\phi)$, $t\in \C^*$, be a family of stable and non-nilpotent $G$-Higgs bundles such that $\nu(\phi,\phi)=0$, with harmonic maps $f_t$ to the symmetric space of $G$.  Then, $$\lim_{|t|\to \infty} \textrm{Ent}(f_t^*\nu)=0.$$
\end{prop}
Note that $\nu(\phi,\phi)=0$ is equivalent to the harmonic map being weakly conformal.
\begin{proof}
    We work in the adjoint representation in order to locally view $\phi$ as an endomorphism valued $1$-form. There is a number $m\leq \textrm{dim}\g$ such that, around every point, $\phi$ has at most $m$ eigen-$1$-forms. Since $\phi$ is holomorphic, on the complement of a discrete set $B$ it locally has $m$ eigen-$1$-forms $\phi_1,\dots, \phi_m$, with multiplicities $k_1,\dots, k_m$. Since $\phi$ is not nilpotent, at least one $\phi_i$ is not identically zero. Note that the locally defined eigen-$1$-forms do not depend on the trivialization, but there is no canonical ordering. On a contractible open subset $U$ of $S-B,$ we define a possibly degenerate metric on $S$ by $$\mu_\phi|_U = \sum_{i=1}^m k_i|\phi_i|^2.$$ Since the $\phi_i$'s are holomorphic, the degenerate points of $\mu_\phi|_U$ are constrained to at most a discrete subset.
    Patching the $\mu_\phi|_U's$ across $S-B$ returns a (possibly degenerate) metric $\mu_\phi$. Since the $\phi_i$'s are holomorphic, $\mu_\phi$ extends smoothly across $B$. From a more geometric perspective, this metric can be interpreted as arising from the harmonic map from the cameral cover, as studied in \cite{SS}.

     If $h$ is the Hermitian metric on $\textrm{ad}P$ induced by the solution to Hitchin's equations, then, for a constant $c$ depending on our normalization of the Killing form $\nu$, $|\phi|_h^2 \geq c\mu_\phi$. Indeed, this follows by, over each point, choosing an $h$-orthonormal frame in which $\phi$ is upper triangular (that is, doing a Schur decomposition on $\phi$). Explicitly, $\phi=\phi_a+\phi_u,$ where $\phi_a$ is diagonal with entries coming from the $\phi_i$'s, and $\phi_u$ is strictly upper triangular. We simply note that $$|\phi|_h^2 = c\mu_{\phi}+|\phi_u|_h^2\geq c\mu_\phi.$$
    
Let $h_t$ be the Hermitian metric coming from the solution to Hitchin's equations at time $t$. Since each $f_t$ is weakly conformal, $f_t^*\nu = |t\phi|_{h_t}^2\geq ct^2\mu_\phi$. Thus, from property (1) and then property (2), $$\textrm{Ent}(f_t^*\nu)\leq \textrm{Ent}(ct^2\mu_\phi)=c^{-1}t^{-2}\textrm{Ent}(\mu_\phi)\to 0$$ as $t\to\infty$.
\end{proof}
One can certainly weaken the hypothesis $\nu(\phi,\phi)=0$. The main reason we include this condition is because it ensures that the harmonic maps are branched immersions, and hence volume entropy always makes sense. We don't need it, but it might be nice to know that for a Coxeter cyclic $G$-Higgs bundle $(P,\phi)$ with Hitchin basepoint $(0,\dots, 0, q_{l})$, $\mu_\phi$ is a scalar multiple of $|q_{l}|^{\frac{2}{l}}.$

\begin{proof}[Proof of Corollary B]
    By Proposition \ref{prop: entropy}, the entropy converges to zero as $|t|\to\infty$. For a weakly conformal map, the pullback metric is the energy density times the conformal metric. By property (1) above and Theorem B, the convergence is monotonic.
\end{proof}

\subsection{Taking $t\to 0$}\label{subsec: t to zero}
Let $(P,\phi)$ be any stable and simple $G$-Higgs bundle over a closed Riemann surface $S$. For each $t\in \C^*$, consider $(P,t\phi),$ let $h_t$ be the Hermitian metric on $\textrm{ad}P$ associated with the solution to Hitchin's equations, and let $f_t:\tilde{S}\to G/K$ be the equivariant harmonic map. Since the Hitchin fibration on the moduli space of $G$-Higgs bundles is proper (see \cite[Theorem 6.11 and Corollary 9.15]{Simpson_properness} and \cite[Theorem 8.1]{Hitchin:1986vp}), upon passing to a subsequence, we can find smooth gauge transformations $s_t$ such that, as $t\to 0,$ $(s_t^*P,ts_t^*\phi)$ converges in the $C^\infty$ sense to a polystable $G$-Higgs bundle $(P_0,\phi_0)$. The $G$-Higgs bundle $(P_0,\phi_0)$ is a $\C^*$-fixed point, and any other choice of subsequence and family of gauge transformations that makes everything converge will lead to a gauge equivalent $G$-Higgs bundle (see \cite[Corollary 9.20]{Simpson_properness}).

Applying the gauge transformation $s_t$ has the effect of conjugating the representation for which $f_t$ is equivariant and translating $f_t$ accordingly. Thus, intrinsic analytic quantities of the $f_t$'s, such as energy densities, are not affected. Alternatively, one can check that a gauge transformation does not alter $|t\phi|_{h_t,\mu}^2=e(f_t)$. It follows from these observations together with standard elliptic regularity results that $s_t^* h_t$ converges in the $C^\infty$ sense to a Hermitian metric $h_0$ that solves Hitchin's equations for $(P_0,\phi_0).$ The metric $h_0$ is in turn equivalent to an equivariant harmonic map $f_0: \tilde{S}\to G/K$, and $e(f_t)$ converges smoothly to $e(f_0)$.

Now, assume that $(P,\phi)$ is also Coxeter cyclic, and not a fixed point of the $\C^*$-action. That is, we're in the setting of Theorem B.
\begin{prop}\label{prop: t=0 case}
    In this setting, away from the zeros of $\phi,$ for all $t\in \C^*,$ 
    \begin{equation}\label{eq: e(f_t)>e(f_0)}
        e(f_t)>e(f_0).
    \end{equation}
Similarly, for all $t\in \C^*,$ $$\textrm{Ent}(f_t^*\nu)<\textrm{Ent}(f_0^*\nu).$$
\end{prop}
\begin{proof}
    By Theorem B, for all $|s|<|t|,$ away from the zeros of $\phi,$ $e(f_t)>e(f_s).$ Since, away from the zeros, $e(f_s)$ monotonically decreases as $|s|\to 0$, the inequality stays strict in the limit, and (\ref{eq: e(f_t)>e(f_0)}) follows. As in Corollary B, the entropy inequality follows from property (1) in the previous subsection.
\end{proof}

Using Remark \ref{rem: constructing gauge transformations}, one can often build explicit gauge transformations. Let $(P,\phi)$ be a stable and simple Coxeter cyclic $G$-Higgs bundle over a closed Riemann surface $S$, which determines line bundles $\{L_\alpha:\alpha\in\mathcal{Z}\}$ and holomorphic sections $\phi_\alpha\in H^0(S,L_\alpha\otimes \mathcal{K}).$ Choose a non-empty subset $I\subset \mathcal{Z}$ and choose non-negative real numbers $k_\alpha$ such that $\sum_{\alpha\in I} n_\alpha=\sum_{\alpha\not \in I} n_\alpha k_\alpha$. By Remark \ref{rem: constructing gauge transformations}, we can find a family of holomorphic gauge transformations $s_t$ of $P$, $t\in \C^*$, such that $s_t^*\phi$ is described by $t^{-1}\phi_\alpha$ for $\alpha \not \in I$ and $t^{k_\alpha} \phi_\alpha$ for $\alpha\in I$ (when $k_\alpha$ is not an integer, we make a choice of $t^{k_\alpha}$). The main case to have in mind is $I=\{-\delta\}$, so that for $\alpha\in \Pi,$ $\phi_\alpha$ becomes $t\phi_\alpha$, and $\phi_{-\delta}$ must be sent to $t^{r}\phi_{-\delta}$. In the general situation, when we take $t\to 0$ along the sequence $(P,s_t^* t\phi),$ the limit is a Coxeter cyclic $G$-Higgs bundle $(P,\phi_0)$, which determines the same line bundles $L_\alpha$ but with sections $(\phi_0)_\alpha=\phi_\alpha$ for $\alpha\in I$ and $(\phi_0)_\alpha=0$ for $\alpha\not \in I$. If $(P,\phi_0)$ is polystable, which, as pointed out in Remark \ref{rem: no phi_i zero}, can be checked using results in \cite[\S 5]{Mcintosh}, then we can declare $s_t$ to be our family of gauge transformations and $(P,\phi_0)$ to be our limit. Note that we can see explicitly that $(P,\phi_0)$ is a fixed point of the $\C^*$-action: $s_t^{-1}$ transforms $(P,\phi_0)$ to $(P,t\phi_0)$.

Finally, let us write out an explicit example. Let $(P,\phi)$, $L_\alpha$, and $\phi_\alpha$ be as above. We assume that, for all $\alpha\in\Pi,$ $\phi_\alpha$ is not the zero section.  Consider the other Coxeter cyclic $G$-Higgs bundle $(P,\phi_0)$ determined by the same line bundles $L_\alpha$, $\alpha\in\mathcal{Z}$, and, for $\alpha\in \Pi,$ the same sections $\phi_\alpha$, and for $\alpha=-\delta,$ the zero section. For convenience, choose an ordering of the simple roots, $\Pi=\{\alpha_1,\dots, \alpha_l\}$, and set $L_i=L_{\alpha_i}$. Let $d_i$ to be the degree of $L_i$. From Proposition 5.4 in \cite{Mcintosh} and Proposition 5.7 in \cite{GPRi}, $(P,\phi_0)$ is polystable if and only if, for every $i,$ 
\begin{equation}\label{eq: polystability condition}
    2-2g\leq d_i, \hspace{1mm} \sum_{j=1}^lR_{ij}d_j<0,
\end{equation}
where $(R_{ij})_{i,j=1}^l$ is the inverse of the matrix whose entries are $\nu(\alpha_i,\alpha_j).$ For each $t\in \C^*$, let $f_t$ be the harmonic map associated with $(P,t\phi)$. As above, we can find gauge transformations $s_t$ such that $s_t^*\phi$ converges to $\phi_0.$ If $h_t$ is the Hermitian metric obtained from solving Hitchin's equations for $(P,t\phi),$ then $s_t^*h_t$ converges to a Hermitian metric for $(P,\phi_0)$, which gives us a harmonic map $f_0.$ Note that, because $(P,\phi_0)$ might be strictly polystable, it might not have a unique solution to Hitchin's equations. The energy densities of all of the different harmonic maps are the same, since they can all be obtained from each other by translating in a flat (this can be seen through \cite{Co}). Proposition \ref{prop: t=0 case} then applies for $(P,t\phi)$, $t\in \C^*$, and $(P,\phi_0)$. We summarize in the theorem below.
\begin{thm}\label{thm: domination result}
Let $(P,\phi)$ be a stable and simple Coxeter cyclic $G$-Higgs bundle whose structure data consists of line bundles $\{L_\alpha:\alpha\in\mathcal{Z}\}$ and holomorphic sections  $\phi_\alpha\in H^0(S,L_\alpha\otimes \mathcal{K}).$ Let $(P,\phi_0)$ be the polystable $\C^*$-fixed point that determines the same line bundles and whose associated sections are $\{\phi_\alpha: \alpha \in \Pi\}$ and the zero section of $L_{-\delta}\otimes \mathcal{K}$.

We assume, after choosing an ordering of the simple roots, that $(P,\phi_0)$ is chosen so that (\ref{eq: polystability condition}) holds. For each $t\in \C^*,$ let $f_t$ be the harmonic map corresponding to $(P,t\phi),$ and for $t=0,$ let $f_0$ be any harmonic map corresponding to $(P,\phi_0).$ Then, for all $t\in \C^*,$ away from the zeros of $\phi,$ $$e(f_t)>e(f_0).$$ As well, $\textrm{Ent}(f_t^*\nu)<\textrm{Ent}(f_0^*\nu)$.
\end{thm}
When $G=\textrm{SL}(n,\C)$, under the defining representation to $\C^n$, $(P,\phi)$ and $(P,\phi_0)$ return Higgs bundles from \cite{Dai2018}, and the analogue of (\ref{eq: polystability condition}) is essentially \cite[Proposition 2.4]{Dai2018}. Theorem \ref{thm: domination result} recovers Theorem 1.3 from \cite{Dai2018}.
 
\section{Extrinsic curvature}\label{sec: curvature}
Here we study extrinsic curvature of harmonic maps and in particular prove Theorems C and C'. Since the $G$-Higgs bundles in the Hitchin section are $G^{\R}$-Higgs bundles, for every case except $\textrm{B}_n$ and $\textrm{D}_n$, Theorem C becomes a special case of Theorem C'.
\subsection{Curvature}\label{subsec: curvature}Let $(P,\phi)$ be a stable $G$-Higgs bundle on a closed surface $S$ of genus $g\geq 2$ with harmonic map $f:\tilde{S}\to G/K$ and harmonic metric $h$. As before, fix a conformal metric $\mu$ on $S$ with area form $\Omega_\mu$, with respect to which we define the root energies $e_\alpha$. Identifying the tangent bundle of $G/K$ with $G/K\times \mathfrak{p}$, the curvature tensor of the Levi-Civita connection is given by $$R(X,Y)Z=-[[X,Y],Z].$$ It follows that the sectional curvature $K_\nu(f_*T\tilde{S})$ of the tangent plane of the harmonic map is
\begin{equation}\label{eq: K formula}
    K_\nu(f_*T\tilde{S}) = -\frac{\nu([\phi,\phi^{*_h}],[\phi,\phi^{*_h}])}{|\phi|_h^2 -|\nu(\phi,\phi)|^2}
\end{equation}
(see, e.g., \cite[pp 16]{Dai2018}).
Since $[\phi,\phi^{*_h}]$ is $h$-self-adjoint, $\nu([\phi,\phi^{*_h}],[\phi,\phi^{*_h}])\geq 0$, and hence the harmonic map is tangent to a flat precisely when $\nu([\phi,\phi^{*_h}],[\phi,\phi^{*_h}])=0.$

Passing to an $\textrm{ad}G$-Higgs bundle, the harmonic map is transported to the $\textrm{ad}G$-symmetric space, with no effect on the sectional curvature. Hence, when $\phi$ is Coxeter cyclic, we see by (\ref{eq: submain-phiphi*}) that
\begin{equation}
    [\phi, \phi^{*_h}]=-2i\Omega_\mu\sum_{\alpha\in\mathcal{Z}} h_\alpha \frac{\nu(\alpha,\alpha)}{2} e_\alpha. 
\end{equation}
Thus,
\begin{equation}\label{eq:curv-sign}
    \abs{\Omega_\mu^{-1}[\phi,\phi^{*_h}]}_\nu^2=4\sum_{\alpha,\beta\in\mathcal{Z}}e_\alpha e_\beta \nu(\alpha, \beta)=4\abs{\sum_{\alpha\in\mathcal{Z}} e_\alpha \alpha}_\nu^2.
\end{equation}
In particular, the curvature $K_\nu(f_*T\tilde{S})$ vanishes if and only if $\sum_{\alpha\in\mathcal{Z}} e_\alpha\alpha=0$. That is, if and only if $e_\alpha=Cn_\alpha$ for some $C\geq 0$ and all $\alpha\in\mathcal{Z}$. Or, if we consider $\tilde{e}_\alpha=e_\alpha\nu(\alpha,\alpha),$ the condition is  $\tilde{e}_\alpha=Cn_\alpha\nu(\alpha,\alpha)$ for some $C\geq 0$.
\begin{remark}\label{rem: products1}
    If $(P,\phi)$ is a direct sum of Higgs bundles $(P_i,\phi_i)$, which occurs for example if we're working with a semisimple group $G=\prod_i G_i$, then any harmonic metric preserves the splitting and hence the numerator in (\ref{eq: K formula}) splits as a sum $-\sum_i\nu([\phi_i,\phi_i^{*_h}],[\phi_i,\phi_i^{*_h}])$. Consequently, once we know Theorem C, we see that the extrinsic curvature is negative as long as one of the Higgs bundles is Hitchin and Coxeter cyclic. 
\end{remark}
\begin{remark}\label{rem: products2}
    Concerning the remark above, one interesting case is $G=\textrm{PSL}(2,\R)^n.$ For each factor, the expression $\nu([\phi, \phi^{*_h}], [\phi, \phi^{*_h}])$ is zero only at points where the harmonic map is not an immersion, so in particular if one factor comes from a Fuchsian representation, then the curvature is negative. We deduce that the analogue of the conjecture from \cite{DL} holds for products of components of the $\textrm{PSL}(2,\R)$ character variety such that at least one factor is Fuchsian. Moreover, we can see that a harmonic map to $(\mathbb{H}^2)^n$ is tangent to a flat exactly when the singular sets of all the component harmonic maps intersect.
\end{remark}
Before moving to Theorems C and C', we address negative curvature for $\C^*$-fixed points. If $(P,\phi)$ is a $\C^*$-fixed point, then $\phi$ is nilpotent, as can be seen using the Hitchin fibration. We then observe that any harmonic map arising from a nilpotent $G$-Higgs bundle has negative extrinsic curvature at immersed points. This observation follows from (\ref{eq: K formula}) and the fact that no non-zero nilpotent endomorphism of a Hermitian vector space commutes with its adjoint (in fact, there is a well-known comparison between the norm of the endomorphism and the norm of the commutator with the adjoint, see \cite[Lemma 3.5]{LM1}).
\subsection{Previous results}\label{sec: previous results}
We can deduce a portion of Theorem C from the results in \cite{DL}. A representation $G\to \textrm{SL}(n,\C)$ with discrete kernel induces an isometric and totally geodesic embedding of symmetric spaces. So it suffices to prove negative curvature for the harmonic map to $\textrm{SL}(n,\C)/\textrm{SU}(n)$ obtained via the map to $G/K$ and such a representation.

Recall that Hitchin representations are defined for groups of adjoint type, but that the representations lift to universal covers (in rank at least $2$). To put ourselves in the setting of \cite{DL}, we lift to covers so that the complex groups corresponding to $\mathrm{A}_n$, $\mathrm{B}_n$, and $\mathrm{C}_n$ are $\textrm{SL}(n+1,\C)$, $\textrm{SO}(2n+1,\C)$, and $\textrm{Sp}(2n,\C)$ respectively, and so that $\mathrm{G}_2$ is a subgroup of $\textrm{SO}(7,\C)$. Let $(P,\phi)$ be a $G$-Higgs bundle and $\sigma: G\to \textrm{SL}(n,\C)$ a representation, inducing an ordinary Higgs bundle $(E,\phi_\sigma).$ If $\sigma$ takes principal $\mathfrak{sl}(2,\C)$'s in the Lie algebra $\g$ to principal $\mathfrak{sl}(2,\C)$'s in $\mathfrak{sl}(n,\C),$ then it sends the Hitchin section for $G$ of \S 3.6 (suitably lifted) to the Hitchin section for $\textrm{SL}(n,\C)$, up to an automorphism of the moduli space $\mathcal{M}_S(G)$. If $\sigma$ is the standard defining representation on a complex vector space, then this is the case for groups of type $\textrm{A}_n,\textrm{B}_n,\textrm{C}_n$, and $\textrm{G}_2$. As well, if $(P,\phi)$ is cyclic via a gauge transformation $s$ of $P,$ then $s$ induces a gauge transformation of $E$ that makes $(E,\phi_\sigma)$ cyclic.

The authors of \cite{DL} prove Theorem C for the linear Higgs bundles corresponding to Coxeter cyclic $G$-Higgs bundles for $G=\textrm{SL}(n,\C)$. Since the Coxeter numbers for groups of type $\textrm{A}_n$ and $\textrm{C}_n$ are $n+1$ and $2n$ respectively, we see that if $(P,\phi)$ above is Coxeter cyclic, then so is $(E,\phi_\sigma).$ Thus, both the $\textrm{A}_n$ and the $\textrm{C}_n$ cases follows from \cite{DL}.

In \cite{DL}, the authors refer to the $\textrm{SL}(n,\C)$-Higgs bundles in the Hitchin section associated with $(0,\dots, 0, q,0)\in\oplus_{i=2}^n H^0(S,\mathcal{K}^i)$ as sub-cyclic Higgs bundles, and they prove their negative curvature conjecture for these Higgs bundles. For Lie groups of type $\textrm{B}_n$ and $\textrm{G}_2,$ such as $\textrm{SO}(2n+1,\C)$ and $\textrm{G}_2$ itself, under the standard linear representation $\sigma$ to $\textrm{SL}(m,\C)$, where $m$ is $2n+1$ in the first case and $7$ in the latter, a Coxeter cyclic $G$-Higgs bundle $(P,\phi)$ becomes a sub-cyclic Higgs bundle $(E,\phi_\sigma)$. Indeed, the sub-cyclic Higgs bundles in the Hitchin component are characterized by being in invariant subspaces of a holomorphic gauge transformation of order $m-1.$ To realize our claim, one just observes that the Coxeter number for the root system of $\textrm{B}_n$ is $2n,$ and for $\textrm{G}_2$ it is $6$, which is $m-1$. Thus, the $\textrm{B}_n$ and $\textrm{G}_2$ cases follow as well.

In the work below we will include the proofs of the $\textrm{A}_n,\textrm{B}_n,\textrm{C}_n$ and $\textrm{G}_2$ cases, using our perspective.
\subsection{Cases $\mathrm{B}_n$ and $\mathrm{D}_n$}
 Assuming $G$ is either of type $\mathrm{B}_n$ or $\mathrm{D}_n$, we prove Theorem C. We label a subset of the affine Dynkin diagram as in Figure \ref{fig:bd}, where $\delta$ is the highest root, and here we label the corresponding functions as $e_{\alpha},$ $e_{\beta},$ and $e_{-\delta}.$  
\par Writing (\ref{eq:bochner-toda}) for this subset of the diagram (which agrees with (\ref{eq:bochner-toda-variant}) if we normalize the Killing form properly), we see that
\begin{align*}
    \frac{1}{2}\Delta_\mu\log e_\alpha=2e_\alpha-e_\beta-K_\mu, \\
    \frac{1}{2}\Delta_\mu\log e_{-\delta}=2e_{-\delta}-e_{\beta}-K_\mu.
\end{align*}
Subtracting these two equations, we get $\Delta_\mu\log\frac{e_\alpha}{e_{-\delta}}=4e_\alpha(1-\frac{e_{-\delta}}{e_\alpha})$. It can be seen from \S 3.6 that, assuming our $G$-Higgs bundle is in the Hitchin section, $e_\alpha>0$ and that $e_{-\delta}\geq 0$ has zeros. Thus the function $\frac{e_\alpha}{e_{-\delta}}$ achieves a minimum on $S$, and by the maximum principle we see that $e_\alpha>e_{-\delta}$. \par However, in both $\mathrm{B}_n$ and $\mathrm{D}_n$ cases, we have $n_\alpha=n_{-\delta}=1$, so in particular $(e_\alpha, e_{-\delta})$ cannot be proportional to $(n_\alpha,n_{-\delta})$. 
\begin{figure}
\dynkin[backwards,labels={,\beta,\alpha,-\delta},scale=1.5]D{***o}\hspace{1mm}$\cdots$
\caption{Root labels for subsets of the extended Dynkin diagrams of type $\mathrm{B_n,D_n}$.}\label{fig:bd}
\end{figure}

\begin{remark}\label{rem: BDextension}
    The argument works for plenty of Coxeter cyclic $G$-Higgs bundles that are not in the Hitchin section. We just need that the divisor of $e_{\alpha}$ is strictly contained in that of $e_{-\delta}$ (note that $\frac{e_\alpha}{e_{-\delta}}$ extends smoothly over zeros of the same multiplicity).
\end{remark}
\begin{remark}
    For $\textrm{B}_n$, the inequality $e_\alpha>e_{-\delta}$ is contained in Lemma 5.4 of \cite{DL}. In fact, that lemma shows more inequalities: $e_\alpha+e_{-\delta}<e_\beta,$ and for any $\beta'$ to the right of $\beta,$ $e_\beta < e_{\beta'}.$ Interestingly, the authors transform the system into a subsystem that involves the term $e_\alpha+e_{-\delta}$, using the tricks $$\Delta_\mu \log (e^f + e^g) = \Delta_\mu \log (e^{f-g}+1) + \Delta_\mu \log e^g$$ and $$\Delta_\mu \log (e^f +1)\geq \frac{e^f}{e^f+1}\Delta_\mu f.$$ 
\end{remark}

\subsection{Theorem C' and the rest of Theorem C}\label{sec: thms C and C'} We first explain the set-up for Theorem C'. If $G$ is of type $\textrm{C}_n,$ $\textrm{F}_4,$ or $\textrm{G}_2$, then the extended Dynkin diagram is an undirected path (with multiple arrows). We label $\tilde{e}_{-\delta}=\tilde{e}_0$, and then for $\alpha\in \Pi,$ we label $\tilde{e}_{\alpha}=\tilde{e}_j$ if $\alpha$ is separated from the lowest root $-\delta$ by $j$ nodes on the extended Dynkin diagram. We can thus write $\{\tilde{e}_\alpha\}_{\alpha\in \mathcal{Z}}=\{\tilde{e}_0,\dots, \tilde{e}_{\ell}\},$ where $\ell=|\mathcal{Z}|.$ 

For Lie groups of type $\textrm{A}_n$ or $\textrm{E}_6$, we consider $G^{\R}$-Higgs bundles. Then, as we mentioned in \S \ref{subsubsec:bochner-toda}, if $\sigma_0$ is the involution from \S \ref{sec: G^R higgs}, we have $\tilde{e}_\alpha = \tilde{e}_{\sigma_0(\alpha)}$ for all $\alpha.$ Note that, when acting on the roots, $\sigma_0$ fixes $-\delta$. We explain below that the Bochner-Toda equations then reduce to equations that are specified by a different Dynkin diagram that is a path. Each node on the diagram is labelled by a function $\tilde{e}_\alpha$, not counting repetitions (that is, $\tilde{e}_\alpha$ and $\tilde{e}_{\sigma_0(\alpha)}$ label the same node). Cutting out the repetitions, the set of functions $\{\tilde{e}_\alpha\}_{\alpha\in \mathcal{Z}}$ can be written $\{\tilde{e}_0,\dots, \tilde{e}_{\ell}\},$ $\ell<|\mathcal{Z}|$ (except, for $\textrm{A}_1$, $\ell=|\mathcal{Z}|$), where, exactly as above, $\tilde{e}_0=\tilde{e}_{-\delta}$, and the other $\tilde{e}_j$'s are labeled by their position relative to $\tilde{e}_0$. 

We assume that $(P,\phi)$ is not a $\C^*$-fixed point. Then, for each $j$, the vanishing locus of $\tilde{e}_j$ coincides with that of a non-zero holomorphic section of a holomorphic line bundle, namely, a corresponding $\phi_\alpha$ from Proposition \ref{prop: structure}. Therefore, $\tilde{e}_j$ has a divisor $D(\tilde{e}_j).$ In the language of the statement of Theorem C', $D_j:= D(\tilde{e}_j).$ We will show the following.
\begin{prop}\label{prop:delta-smallest}
    Suppose that $G$ is of type $ \textrm{C}_n,\textrm{F}_4,$ or $\textrm{G}_2$ and that $(P,\phi)$ is a Coxeter cyclic $G$-Higgs bundle, or that $G$ is of type $\textrm{A}_n$ or $\textrm{E}_6$, and that $(P,\phi)$ is a Coxeter cyclic $G^{\R}$-Higgs bundle. We assume that neither of the two $G$-Higgs bundles are $\C^*$-fixed points. Suppose that
    $$D(\tilde{e}_0)>D(\tilde{e}_1) \geq \dots \geq D(\tilde{e}_{\ell}) = 0.$$
    Then,
    \begin{align*}
        \tilde{e}_{-\delta}<\frac{\tilde{e}_\alpha}{n_\alpha}\text{ pointwise, for all }\alpha\in \Pi.
    \end{align*}
\end{prop}
As can be read off of the extended Dynkin diagrams, $\nu(\delta,\delta)\geq \nu(\alpha,\alpha)$ for all $\alpha\in \Pi.$ Thus, $$\frac{\tilde{e}_{-\delta}}{\nu(\delta,\delta)}< \frac{\tilde{e}_\alpha}{n_\alpha\nu(\alpha,\alpha)}$$ pointwise, for all $\alpha\in \Pi$. Hence, by (\ref{eq:curv-sign}), Proposition \ref{prop:delta-smallest} immediately implies Theorem C'. If $(P,\phi)$ is in the Hitchin section, then it is a $G^{\R}$-Higgs bundle, and $D(\tilde{e}_0)$ is the divisor of the holomorphic differential that we obtain via the Hitchin fibration, while $D(\tilde{e}_1)=\dots = D(\tilde{e}_{\ell}) =0.$ Thus, the remaining cases in Theorem C represent a particular example that falls into Theorem C'. In the remainder of this section, we explain how the equations reduce for types $\textrm{A}_n$ and $\textrm{E}_6,$ and then we give the proof of Proposition \ref{prop:delta-smallest}.

We begin with the equations. Pictorially, as also explained in \cite[\S 2.2.4]{B}, in the cases $\textrm{A}_n$ and $\textrm{E}_6$, the involution $\sigma_0$ encodes a folding symmetry of the (ordinary) Dynkin diagram, shown in the second column of Table \ref{table:fold}. 
In this second column, we've drawn the Dynkin diagram with nodes that are coloured black, and then shown how to complete to the extended Dynkin diagram via adjoining a node (not coloured black) corresponding to the lowest root $-\delta$. 
For each diagram in this second column, energies of roots connected by grey lines are equal. By collapsing the Dynkin diagram and then adding a node for the lowest root, we get a different (variant) Bochner--Toda system associated with a generalized Cartan matrix of affine type. Such matrices correspond to Dynkin diagrams of affine type, shown in the final column of Table \ref{table:fold}. The third column of Table \ref{table:fold} shows the names of these diagrams from the list of affine Dynkin diagrams in \cite[\S 15.20, pp. 354]{Carter2005}.


\par After this collapse, for all of the cases of Theorem C', Hitchin's equations indeed reduce to Bochner--Toda systems associated with affine Dynkin diagrams that are undirected paths (with multiple arrows). We show the relevant analytic result for such systems in \S\ref{subsubsec:bt-line}. In \S\ref{subsubsec:odd-a-e6} and \S\ref{subsubsec:even-a}, we explain how to obtain the diagrams in the final column of Table \ref{table:fold} from the diagrams in the second column.

\begin{table}
    \begin{tabular}{ c|c|c|c } 
     Type & Diagram & Folded type  & Folded diagram \\
     \hline
     $\tilde{\mathrm{A}}_{2n-1}$ & \dynkin[fold,extended]A{**...***...**} & $\tilde{\mathrm{C}}_n^t$ & \dynkin[reverse arrows,extended]C{**...**} \\ 
     $\tilde{\mathrm{A}}_{2n}$ & \dynkin[fold,extended]A{**...**...**} & $\tilde{\mathrm{C}}_n'$ & \dynkin[affine mark=o]A[2]{**...***}\\ 
     $\tilde{\mathrm{E}}_6$ & \dynkin[extended,fold,backwards]E6 & $\tilde{\mathrm{F}}_4^t$ & \dynkin[reverse arrows,extended]F4\\ 
    \end{tabular}
    \caption{Folding the Dynkin diagrams}\label{table:fold}
\end{table}
\FloatBarrier
\subsubsection{The odd $\mathrm{A}_n$ and $\mathrm{E}_6$}\label{subsubsec:odd-a-e6}
Consider the diagram of Figure 2, describing $\textrm{E}_6$,
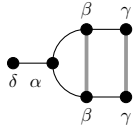
\begin{figure}
    \dynkin[fold,labels={\gamma,\eta,\beta,\alpha,\beta,\gamma},scale=1.5,backwards]E6
    \caption{Dynkin diagram for $\textrm{E}_6$.}
\end{figure}
and assume we have a Bochner--Toda system associated with this Dynkin diagram, with the extra imposed symmetries as indicated by the grey lines, and associated functions $\tilde{e}_\delta$, $\tilde{e}_\alpha$, etc. The equations of (\ref{eq:bochner-toda-variant}) for $\alpha$ and $\beta$ are then 
\begin{gather*}
    \frac{1}{2}\Delta_\mu\log \tilde{e}_\alpha=2\tilde{e}_{\alpha}-2\tilde{e}_{\beta}-\tilde{e}_\eta-1,\\
    \frac{1}{2}\Delta_\mu\log \tilde{e}_\beta=2\tilde{e}_{\beta}-\tilde{e}_{\alpha}-\tilde{e}_{\gamma}-1.
\end{gather*}
This is identical to a Bochner--Toda system where $a_{\alpha\beta}=-2$ and $a_{\beta\alpha}=-1$, i.e., one that corresponds to Figure 3.
\begin{figure}
    \dynkin[labels={\eta,\alpha,\beta,\gamma},scale=1.5,reverse arrows] F4 $\cdots$
    \caption{A portion of a folded diagram.}
\end{figure}
The presence of $\eta$ does not impact this folding. In particular, the folds for $\tilde{\mathrm{A}}_{2n-1}$ and $\tilde{\mathrm{E}}_6$ are all of this form, and we've implicitly described the folding for $\tilde{\mathrm{A}}_{2n-1}$ as well.
\subsubsection{The even $\mathrm{A}_n$}\label{subsubsec:even-a} We show the case $n=6$ in Figure 4.
\begin{figure}
    \dynkin[fold,labels={\alpha,\beta,\gamma,\gamma,\beta,\alpha},scale=1.5] A6
    \caption{Diagram for $\textrm{A}_n$ for $n=6$.}
\end{figure}
The system (\ref{eq:bochner-toda-variant}) at $\gamma$ and $\beta$ is 
\begin{gather*}
    \frac{1}{2}\Delta_\mu\log \tilde{e}_\gamma=2\tilde{e}_\gamma-\tilde{e}_\beta-\tilde{e}_{\gamma}-1=\tilde{e}_{\gamma}-\tilde{e}_{\beta}-1, \\
    \frac{1}{2}\Delta_\mu \log \tilde{e}_\beta=2\tilde{e}_{\beta}-\tilde{e}_\gamma-\tilde{e}_{\alpha}.
\end{gather*}
Thus, setting $\tilde{e}_\gamma'=\frac{e_\gamma}{2}$, we have 
\begin{gather*}
    \frac{1}{2}\Delta_\mu\log\tilde{e}_\gamma'=2\tilde{e}_\gamma'-\tilde{e}_\beta, \\
    \frac{1}{2}\Delta_\mu \log \tilde{e}_\beta=2\tilde{e}_\beta-2\tilde{e}_\gamma'-\tilde{e}_{\alpha}.
\end{gather*}
This is equivalent to a Bochner--Toda system with $a_{\tilde{\gamma}\beta}=-1, a_{\beta\tilde{\gamma}}=-2$. This is achieved by the diagram with a double edge directed from $\tilde{\gamma}$ to $\beta$.
\subsubsection{Bochner--Toda system associated with a path}\label{subsubsec:bt-line}
By \S\ref{subsubsec:odd-a-e6}, \S\ref{subsubsec:even-a}. and Table \ref{table:fold}, it suffices to show Proposition \ref{prop:delta-smallest} for Bochner--Toda systems associated with $\tilde{\rm C}_n^t, \tilde{\rm C}_n', \tilde{\rm C}_n, \tilde{\rm F}_4,$ $\tilde{\rm G}_2$ or $\tilde{\rm F}_4^t$. All of these correspond to generalized Cartan matrices of affine type, and after forgetting edge multiplicity, the Dynkin diagram becomes an undirected path. 
\par Suppose, therefore, that we are in the following general setting: let $A$ be a generalized Cartan matrix of affine type, whose underlying Dynkin diagram is a path. Recall that the fact that $A$ is of affine type means that there exists a unique vector $u\in\mathbb{R}_{>0}^n$ such that $Au=0$ \cite[pp. 337]{Carter2005}.
\begin{lem}\label{lm:order}
    Let $w:S\to\mathbb{R}_{\geq 0}^n$ be a smooth function such that 
    \begin{align*}
        \frac{1}{2}\Delta_\mu \log w_i=\sum_{j=1}^nA_{ij}w_j - 1\text{ for all }i=1,2,...,n.
    \end{align*}
Suppose that every $w_i$ has a discrete (possibly empty) zero set, and that the zero set of $w_1$ is non-empty. Suppose further that, if $D(w_i)$ is the divisor describing the zero set with multiplicity of $e_i$, 
       $$D(w_1)>D(w_2)\geq D(w_3)\geq \dots \geq D(w_n)=0.$$
Then we have a pointwise estimate
    \begin{align*}
        \frac{w_1}{u_1}\leq \frac{w_2}{u_2}\leq...\leq \frac{w_n}{u_n},
    \end{align*}
    and each inequality $\frac{w_i}{u_i}\leq \frac{w_{i+1}}{u_{i+1}}$ fails to be strict only at points where $w_{i+1}$ vanishes.
\end{lem}  
For the Coxeter cyclic $G$-Higgs bundles of Theorem C', $w_1$ is $\tilde{e}_{-\delta}=\tilde{e}_0$, which necessarily has zeros, while the other $w_i$'s correspond to (rescaled) energies of simple roots. Thus, Lemma \ref{lm:order} immediately implies Proposition \ref{prop:delta-smallest} in the cases when there is no folding, i.e., for $\mathrm{C}_n, \mathrm{F}_4$ or $\mathrm{G}_2$. For the other cases ($\mathrm{A}_n$ or $\mathrm{E}_6$), it suffices to observe that $n_{\nu(\alpha)}=n_\alpha$, where $\nu:\Pi\to\Pi$ is the involution defined by the grey lines in Table \ref{table:fold}, and that the vector $u$ is obtained by identifying the coordinates $n\in\mathbb{Z}^{\mathcal{Z}}$ along $\nu$.
\begin{remark}
    In the case of $\tilde{\rm A}_{2n}$, since $\tilde{e}_\gamma'=\frac{\tilde{e}_\gamma}{2}$, we should also set $n_\gamma'=\frac{n_\gamma}{2}$. Then $\frac{\tilde{e}_\gamma'}{n_\gamma'}=\frac{\tilde{e}_\gamma}{n_\gamma}$, so Lemma \ref{lm:order} still implies Proposition \ref{prop:delta-smallest}.
\end{remark}
\begin{proof}[Proof of Lemma \ref{lm:order}]
    Write $f_j=\frac{w_j}{u_j}$ and $B_{ij}=A_{ij}u_j$. Then, on the complement of the zero set of $f_1$,
    \begin{align*}
        \frac{1}{2}\Delta_\mu\log f_i=\sum_{j=1}^n B_{ij}f_j - 1\text{ for all }i=1,2,...,n.
    \end{align*}
    By assumption, we have $\sum_{j=1}^n B_{ij}=0$. Since the Dynkin diagram of $A$ is a path, it follows that 
    \begin{align*}
        \sign{B_{ij}}=\left\{\begin{matrix*}[l]
            1 & \text{if }i=j, \\
            -1 & \text{if }\abs{i-j}=1, \\
            0 & \text{if }\abs{i-j}\geq 2.
        \end{matrix*}\right.
    \end{align*}
    Write $B_{i,i+1}=-U_i$ and $B_{i,i-1}=-L_i$, so in particular $B_{ii}=L_i+U_i$. Thus $U_i, L_i>0$. We now have, still away from the zeros of $f_1,$
    \begin{align*}
        \frac{1}{2}\Delta_\mu\log\frac{f_{i+1}}{f_i}&=(L_{i+1}+U_{i+1}+U_i)f_{i+1}-U_{i+1}f_{i+2}-(L_{i+1}+L_i+U_i)f_i+L_i f_{i-1}\\
        &=-L_i(f_i-f_{i-1})+(L_{i+1}+U_i)(f_{i+1}-f_i)-U_{i+1}(f_{i+2}-f_{i+1}).
    \end{align*}
    Setting $v_i=\log\frac{f_{i+1}}{f_i}$ and $v=(v_1,\dots, v_{n-1})$, we see that $\Delta_\mu v=2 CDv$, where $D$ is the diagonal matrix with entries 
    \begin{align*}
        D_{ii}=\frac{f_{i+1}-f_i}{\log f_{i+1}-\log f_i},
    \end{align*}
    and 
    \begin{align*}
        C_{ij}=\left\{\begin{matrix*}[l]
            L_{i+1}+U_i &\text{ if }i=j, \\
            -U_{i+1} &\text{ if }j=i+1, \\
            -L_i &\text{ if }j=i-1.
        \end{matrix*}\right.
    \end{align*}
    The matrix $C_{ij}$ is cooperative and has vanishing column sums, and is thus column-diagonally dominant. Therefore, so is $2CD$. The system for $\Delta_\mu v$ is well-defined away from points such that $w_{i}$ has a higher order of vanishing than $w_{i+1}$. As we approach these points (if any), the corresponding function $v_i$ tends to $+\infty.$ Since $v_1=\log\frac{f_2}{f_1}$ is unbounded from above (since $w_1$, and therefore $f_1$, vanishes somewhere), condition (2) from Theorem \ref{thm: DL max} is triggered. It follows that $v_i>0$ for all $i=1,2,...,n-1$, and moreover that $f_{i+1}>f_i$. The result is shown.
\end{proof}


\appendix

\section{Maximum principles for elliptic systems}
Here we generalize the methods of Dai--Li that prove Theorem \ref{thm: DL max} (i.e. Lemma 3.1 in \cite{Dai2018}) to get a general result that can be used to produce more maximum principles, Theorem \ref{thm:mp}. Theorem \ref{thm:mp} is not used in the proof of our main theorem, but we however believe it helps further explain Theorem \ref{thm: DL max}. 
\par We show in particular how to use Theorem \ref{thm:mp} to prove a slight generalization of case (3) of Dai-Li's maximum principle, namely Theorem \ref{thm:dai-li-gen} below. Since Dai-Li's maximum principle is an important ingredient in the proofs of the main theorems, we thought it worthwhile to include this content.

\subsection{Graph neighbourhood maximum principle}
Denote by $\langle\cdot,\cdot\rangle$ the standard Euclidean inner product on $\mathbb{R}^n$, and by $A^*$ the adjoint of the linear map $A:\mathbb{R}^n\to\mathbb{R}^m$ with respect to this inner product. Denote by $\mathbb{R}^{n\times n}$ the space of endomorphisms $\mathbb{R}^n\to\mathbb{R}^n$.
\begin{defn}
   Let $X$ be a manifold and $A:X\to\mathbb{R}^{n\times n}$ be a smooth map. Given vectors $v_0,v_1,v_2,...,v_k\in\mathbb{R}^{n}$, we say that $v_0$ is $A$-weaker than $(v_1,v_2,...,v_k)$ if there exist smooth functions $\lambda_i:X\to\mathbb{R}_{\geq 0}$ for $i=1,2,...,k$, such that 
    \begin{align*}
       A(x)v_0=\sum_{i=1}^k \lambda_i(x)(v_0-v_i)\text{ for all }x\in X.
    \end{align*}
    Moreover, we say that $v_0$ is minimally $A$-weaker than $(v_1,v_2,...,v_k)$ if no $\lambda_i$ vanishes identically.
\end{defn}
\begin{thm}\label{thm:mp}
    Let $X$ be a closed Riemannian manifold, and let $A:X\to\mathbb{R}^{n\times n}$ be a smooth map. Suppose that $u:X\to\mathbb{R}^n$ is a solution to $$\Delta u=Au.$$ Let $S\subset\mathbb{R}^n$ be a finite set of vectors with a distinguished subset $ S_\partial\subset S$, and suppose that $\Gamma$ is a directed graph on $S$, such that
    \begin{enumerate}
        \item\label{item:structure} for any $v\in S\setminus S_\partial$, the vector $v$ is minimally $A^*$-weaker than its (directed) neighbourhood, and
        \item\label{item:reachable} for any $v\in S\setminus S_\partial$, there exists a vertex in $S_\partial$ that can be reached from $v$.
    \end{enumerate}
    Then 
    \begin{align}\label{eq:mp-main}
        \min_{v\in S\setminus S_\partial}\min_X\langle v,u\rangle\geq \min_{v\in S_\partial}\min_X\langle v,u\rangle.
    \end{align}
    If equality holds in (\ref{eq:mp-main}), and if we let $v_0\in S\setminus S_\partial$ be any vector achieving the minimum in the left-hand side of (\ref{eq:mp-main}), then the function $\langle v_0, u\rangle$ is constant.
\end{thm}
Theorem \ref{thm:mp} readily follows from Lemma \ref{lm:mp-directional} below, whose proof we defer to the next subsection. 
\begin{lem}\label{lm:mp-directional}
    Let $X, A, u$ be as in Theorem \ref{thm:mp}. Then, if $v_0,v_1,v_2,...,v_k\in\mathbb{R}^{n}$ are such that $v_0$ is minimally $A^*$-weaker than $(v_1,v_2,...,v_k)$, then  
    \begin{align}\label{eq:ineq}
        \min_X\langle v_0,u\rangle\geq \min_{1\leq i\leq k}\min_X\langle v_i, u\rangle.
    \end{align}
    Moreover, if equality holds in (\ref{eq:ineq}), then $\min_X\langle v_0,u\rangle\geq\min_X \langle v_j, u\rangle$ for all $0\leq j\leq k$ and $\langle v_0, u\rangle$ is constant.
\end{lem}
We now show Theorem \ref{thm:mp} assuming Lemma \ref{lm:mp-directional}. For any $v\in S$, define $b_v=\min_X \langle v, u\rangle$, and set $b_S=\min_{v\in S}b_v$ and $b_{S_\partial}=\min_{v\in S_\partial}b_v$. Note that clearly $b_S\leq b_{S_\partial}$, and we need to show equality.
\begin{claim}\label{claim:minima-spread}
    If $v\in S\setminus S_\partial$ is such that $b_v=b_S$, then $b_w=b_S$ for all $w\in N_\Gamma^+(v)$.
\end{claim}
\begin{proof}
    By Lemma \ref{lm:mp-directional}, we have 
    \begin{align*}
        b_v\geq \min_{w\in N^+_\Gamma(v)} b_w.
    \end{align*}
    Since $b_v=b_S\leq b_w$ for all $w\in S$, we see that $b_v=\min_{w\in N_\Gamma^+(v)}b_w$. Hence equality holds in (\ref{eq:ineq}) from Lemma \ref{lm:mp-directional}, and therefore $b_w\leq b_v=b_S$ for all $w\in N_\Gamma^+(v)$, as desired. However. by definition we have $b_S\leq b_w$, so $b_w=b_S$ for $w\in N_\Gamma^+(v)$.
\end{proof}
Returning to the theorem, let $v_0\in S$ be such that $b_S=b_{v_0}$. We split the proof into two cases.
\begin{enumerate}
    \item  If $v_0\in S_\partial$, then $b_S=b_{S_\partial}$ and there is nothing to prove. 
    \item  If $v_0\not\in S_\partial$, by assumption (\ref{item:reachable}) in Theorem \ref{thm:mp}, there exists a sequence $v_0, v_1, v_2,...,v_k\in S\setminus S_\partial$ and $v_{k+1}\in S_\partial$, such that $v_{i+1}\in N_\Gamma^+(v_i)$ for $i=0,1,2,...,{k}$. By induction and Claim \ref{claim:minima-spread}, we see that 
    \begin{align*}
        b_S=b_{v_0}=b_{v_1}=...=b_{v_{k+1}}\geq b_{S_\partial},
    \end{align*}
    and hence $b_S=b_{S_\partial}$, as desired.
\end{enumerate}
Suppose now that $b_{v_0}=b_{S_\partial}$, i.e., that equality holds in (\ref{eq:mp-main}). Then we have $b_{v_0}=\min_{w\in N_\Gamma^+(v_0)} b_w$, so by the equality case of Lemma \ref{lm:mp-directional}, it follows that $\langle v_0, u\rangle$ is constant.

\subsection{Proof of Lemma \ref{lm:mp-directional}}   
\begin{proof}[Proof of Lemma \ref{lm:mp-directional}]
We have 
    \begin{align*}
        \Delta\langle v_0,u\rangle=\langle v_0, Au\rangle=\langle A^*v_0, u\rangle.
    \end{align*}By assumption, there exist smooth functions $\lambda_1, \lambda_2,...,\lambda_k:X\to\mathbb{R}_{\geq 0}$ such that 
    \begin{align*}
        A^*v_0=\sum_{i=1}^k \lambda_i(v_0-v_i).
    \end{align*}
    Thus 
    \begin{align}
        \Delta \langle v_0,u\rangle&=\sum_{i=1}^k \lambda_i \langle v_0-v_i, u\rangle=\left(\sum_{i=1}^k \lambda_i\right)\left(\langle v_0, u\rangle-\sum_{j=1}^k \frac{\lambda_j}{\sum_{i=1}^k\lambda_i} \langle v_i, u\rangle \right)\nonumber\\
        &\leq \left(\sum_{i=1}^k \lambda_i\right)\left(\langle v_0, u\rangle-\min_{1\leq j\leq k}\langle v_j, u\rangle\right)\label{eq:ineq-1}\\ &\leq\left(\sum_{i=1}^k \lambda_i\right)\left(\langle v_0, u\rangle-\min_{1\leq j\leq k}\min_X \langle v_j, u\rangle\right)\label{eq:eq-2}.
    \end{align}
    Since no $\lambda_i$ vanishes identically, neither does $\sum_{i=1}^k\lambda_i$. Applying the maximum principle to the function $\langle v_0,u\rangle-\min_{1\leq j\leq k}\min_X \langle v_i, u\rangle$, we see that it is non-negative, and (\ref{eq:ineq}) is shown.
    \par Assume now that we have equality in (\ref{eq:ineq}). By the strong maximum principle \cite[Theorem 3.5, pp. 35]{Gilbarg2001}, $\langle v_0,u\rangle\equiv\min_{1\leq j\leq k}\min_X\langle v_j,u\rangle$. Thus equality holds in both (\ref{eq:ineq-1}) and (\ref{eq:eq-2}). Since equality holds in (\ref{eq:eq-2}), we have $\min_{1\leq j\leq k}\langle v_j, u\rangle\equiv \min_{1\leq j\leq k}\min_X \langle v_j, u\rangle=:b$. Since equality holds in (\ref{eq:ineq-1}), we see that, for any $j\in \{1,2,...,k\}$, 
    \begin{align*}
        \langle v_j,u\rangle=\min_{1\leq i\leq k}\langle v_i, u\rangle=b\text{ on the open set }\lambda_j^{-1}(\mathbb{R}_+). 
    \end{align*}
    By minimality, $\lambda_j^{-1}(\mathbb{R}_+)$ is non-empty, so $\min_X\langle v_j, u\rangle\leq b$.
    \end{proof}
\subsection{Recovering Dai-Li's maximum principle}
By choosing suitable parameters for Theorem \ref{thm:mp}, we prove the following.
\begin{thm}\label{thm:dai-li-gen}
    Let $X$ be a closed Riemannian manifold, and let $c_{ij}:X\to\mathbb{R}$ be smooth functions indexed by $1\leq i,j\leq n$, such that 
    \begin{enumerate}
        \item(cooperative) $c_{ii}\geq 0$ and $c_{ij}\leq 0$ for $i\neq j$, 
        \item(column-diagonally dominant) $\sum_{i=1}^n c_{ij}\geq 0$ for all $1\leq j\leq n$, and 
        \item(fully coupled) there is no partition $\{1,2,...,n\}=A\cup B$ such that $c_{ij}=0$ for $i\in A, j\in B$.  
    \end{enumerate}
    Suppose that $u_i:X\to\mathbb{R}$ for $i=1,\dots, n$ solve the system 
    \begin{align*}
        \Delta u_i=\sum_{j=1}^n c_{ij}u_j\text{ for }i=1,2,...,n.
    \end{align*}
    Now suppose that for some constants $\nu_i>0$ we have $\sum_{i=1}^n \nu_i u_i\geq 0$. Then either $u_i\equiv 0$ for all $i$, or $u_i>0$ for all $i$. 
\end{thm}
When $\nu_i=1$ for all $i\in\{1,2,...,n\}$, this is case (3) of Theorem \ref{thm: DL max}.

The remainder of this subsection is devoted to the proof of Theorem \ref{thm:dai-li-gen}. Let $C$ be the $n\times n$ matrix with entries $c_{ij}$. For a set $A\subseteq\{1,2,...,n\}$, we denote by $e_A\in\mathbb{R}^n$ the vector given by  
\begin{align*}
    (e_A)_i=\left\{\begin{matrix}
        1 & \text{ if }i\in A,\\
        0 & \text{ if }i\not\in A.
    \end{matrix}\right.
\end{align*}
We will construct a graph $\Gamma$ on the vertex set $S=\{e_A:A\subseteq\{1,2,...,n\}\}\cup\{K\nu\}$, where $\nu=(\nu_1,\nu_2,...,\nu_n)\in\mathbb{R}_+^n$ and $K>\frac{1}{\min_i\nu_i}$. We will then apply Theorem \ref{thm:mp} with $S_\partial=\{e_0, K\nu\}=\{0,K\nu\}$.
\par Let $\bar{\Gamma}$ be the graph defined as follows: for any proper subset $A\subset\{1,2,...,n\}$, we let 
\begin{align*}
    N_{\bar{\Gamma}}^+(e_A)=\{e_{A\cup\{j\}}:j\not\in A\}\cup\{e_{A\setminus\{j\}}:j\in A\},
\end{align*}
and
\begin{align*}
    N_{\bar{\Gamma}}^+(e_{\{1,2,...,n\}})=\{e_{\{1,2,...,n\}\setminus\{j\}}: j\in \{1,2,...,n\}\}\cup\{K\nu\}.
\end{align*}
\begin{claim}\label{claim:only}
    For any proper subset $A\subseteq\{1,2,...,n\}$, there exists a subset $\mathcal{N}_A\subset N_{\bar{\Gamma}}^+(e_A)$, such that
    \begin{enumerate}
        \item\label{item:minimal} the vector $e_A$ is minimally $C^*$-weaker than $\mathcal{N}_A$, and 
        \item\label{item:points-to-smaller} there exists some subset $B\subsetneq A$ such that $e_B\in \mathcal{N}_A$.
    \end{enumerate}
    Furthermore, $e_{\{1,2,...,n\}}$ is minimally $C^*$-weaker than $N_{\bar{\Gamma}}^+(e_{\{1,2,...,n\}})$.
\end{claim}
\begin{proof}
    We first show the claim for $e_{\{1,2,...,n\}}$. Observe that 
    \begin{align*}
        \left(C^*e_{\{1,2,...,n\}}\right)_i=\sum_{j=1}^n C_{ji}\geq 0.
    \end{align*}
    Note that $e_{\{1,2,...,n\}}-K\nu<0$, so we can write 
    \begin{align*}
        C^*e_{\{1,2,...,n\}}=(e_{\{1,2,...,n\}}-K\nu)+\sum_{i=1}^n \lambda_i\left(e_{\{1,2,...,n\}}-e_{\{1,2,...,n\}\setminus\{i\}}\right),
    \end{align*}
    where $\lambda=K\nu-1+C^*e_{\{1,2,...,n\}}>0$. It follows immediately that $e_{\{1,2,...,n\}}$ is minimally $C^*$-weaker than $N_{\bar{\Gamma}}^+(e_{\{1,2,...,n\}})$.
    \par We now turn to other vertices of $\bar{\Gamma}$. Fix a proper subset $A\subset\{1,2,...,n\}$. Note that, since $C$ is cooperative, we have 
    \begin{align*}
        (C^*e_A)_i=\sum_{j\in A} C_{ji}\leq 0\text{ for }i\not\in A.
    \end{align*}
    Similarly, for $i\in A$, we have 
    \begin{align}\label{eq:check-sign-A}
        (C^*e_A)_i=\sum_{j\in A}C_{ji}=\sum_{j=1}^n C_{ji} - \sum_{j\not\in A} C_{ji}\geq 0.
    \end{align}
    Thus \begin{align*}
        C^*e_A=\sum_{i\in A} \lambda_i (e_A-e_{A\setminus\{i\}})+\sum_{i\not\in A}\lambda_i (e_A-e_{A\cup\{i\}}),
    \end{align*}
    where 
    \begin{align*}
        \lambda_i=\left\{\begin{matrix}
            \sum_{j\in A} C_{ji} & \text{ if }i\in A,\\
            -\sum_{j\in A} C_{ji} & \text{ if }i\not\in A.
        \end{matrix}\right.
    \end{align*}
    If we let $\mathcal{N}_A=\{A\cup\{i\}:i\not\in A\text{ and }\lambda_i\not\equiv 0\}\cup\{A\setminus\{i\}:i\in A\text{ and }\lambda_i\not\equiv 0\}$, then $e_A$ is minimally $C^*$-weaker than $\mathcal{N}_A$, so (\ref{item:minimal}) is shown.
    \par It remains to check (\ref{item:points-to-smaller}). Assume that (\ref{item:points-to-smaller}) fails, i.e., that $\lambda_i\equiv 0$ for all $i\in A$. Therefore 
    \begin{align*}
        \sum_{j\in A}C_{ji}=0\text{ for }i\in A.
    \end{align*}
    Then equality must hold in (\ref{eq:check-sign-A}), so 
    \begin{align*}
        \sum_{j=1}^n C_{ji}=0\text{ and }\sum_{j\not\in A}C_{ji}=0\text{ for }i\in A.
    \end{align*}
    In particular, $C_{ji}=0$ for $i\in A, j\not\in A$. This is a contradiction as $C$ is fully coupled, which shows (\ref{item:points-to-smaller}).
\end{proof}
Now let $\Gamma$ be the subgraph of $\bar{\Gamma}$ such that $N_\Gamma^+(e_A)=\mathcal{N}_A$ for all subsets $A\subseteq\{1,2,...,n\}$. Assumption (\ref{item:structure}) in Theorem \ref{thm:mp} holds by Claim \ref{claim:only}(\ref{item:minimal}), and assumption (\ref{item:reachable}) in Theorem \ref{thm:mp} follows from Claim \ref{claim:only}(\ref{item:points-to-smaller}). Therefore 
\begin{align*}
    u_i=\langle e_{\{i\}}, u\rangle\geq\min_{v\in \{0, K\nu\}}\min_X \langle v,u\rangle =\min\left(0,\min_X K\langle \nu, u\rangle\right)= 0.
\end{align*}
Thus we have shown that $u_i\geq 0$ for all $i$. It remains to check that either all $u_i$ are positive, or they all vanish identically. This will follow from the next claim.
\begin{claim}\label{claim:vanishing-ui}
    \begin{enumerate}
        \item If $u_i$ vanishes somewhere, then $u_i\equiv 0$.
        \item If $u_i\equiv 0$ and $c_{ij}\not\equiv 0$ for some $j$, then $u_j\equiv 0$. 
    \end{enumerate}
\end{claim}
\begin{proof}
\begin{enumerate}
    \item 
Suppose that $u_i$ vanishes somewhere on $X$. Then equality is achieved in Theorem \ref{thm:mp} for $e_{\{i\}}$. In particular, we see that $u_i$ is constant, and hence $u_i\equiv 0$. 
\item 
    Suppose that $u_i\equiv 0$. Then $0=\Delta u_i=c_{ii} u_i+\sum_{j\neq i} c_{ij} u_j=\sum_{j\neq i} c_{ij}u_j\leq 0$, since $c_{ij}u_j\leq 0$. Thus $c_{ij}u_j\equiv 0$ for all $j\neq i$. If $c_{ij}\not\equiv 0$, then $u_j$ vanishes somewhere on $X$, so by (1), we have $u_j\equiv 0$.
\end{enumerate}
\end{proof}
Let $A$ be the set of indices $i$ such that $u_i$ vanishes somewhere. Note that for any $i\in A, j\in\{1,2,...,n\}\setminus A$, we have $c_{ij}\equiv 0$ by Claim \ref{claim:vanishing-ui}. Since the matrix $C$ is fully coupled, it follows that $A$ is either empty or $\{1,2,...,n\}$. In the former case, we have $u_i>0$ for all $i$. In the latter case, we have $u_i\equiv 0$ for all $i$ by Claim \ref{claim:vanishing-ui}(2). This concludes the proof of Theorem \ref{thm:dai-li-gen}.

\bibliographystyle{plain}
\bibliography{bibliography}

\begin{thebibliography}{10}

\bibitem{CLGP}
Luis Alvarez-Consul and Oscar Garcia-Prada.
\newblock {Hitchin-Kobayashi correspondence, quivers, and vortices}.
\newblock {\em Commun. Math. Phys.}, 238:1--33, 2003.

\bibitem{Baraglia2010CyclicHB}
D.~Baraglia.
\newblock Cyclic higgs bundles and the affine toda equations.
\newblock {\em Geometriae Dedicata}, 174:25--42, 2010.

\bibitem{B}
David Baraglia.
\newblock G2 geometry and integrable systems, 2010.

\bibitem{BGM}
Steven~B. Bradlow, Oscar Garc\'{\i}a-Prada, and Ignasi Mundet~i Riera.
\newblock Relative {H}itchin-{K}obayashi correspondences for principal pairs.
\newblock {\em Q. J. Math.}, 54(2):171--208, 2003.

\bibitem{Carter2005}
Roger~William Carter.
\newblock {\em {Lie algebras of finite and affine type}}.
\newblock Number~96 in Cambridge studies in advanced mathematics. Cambridge University Press, Cambridge, 2005.

\bibitem{C}
Brian Collier.
\newblock Finite order automorphisms of higgs bundles: theory and application.
\newblock {\em PhD thesis}, 2016.

\bibitem{CL}
Brian Collier and Qiongling Li.
\newblock Asymptotics of {H}iggs bundles in the {H}itchin component.
\newblock {\em Adv. Math.}, 307:488--558, 2017.

\bibitem{CB}
Brian Collier and J\'er\'emy Toulisse.
\newblock Holomorphic curves in the 6-pseudosphere and cyclic surfaces.
\newblock {\em Trans. Amer. Math. Soc.}, 377(9):6465--6514, 2024.

\bibitem{Co}
Kevin Corlette.
\newblock Flat {$G$}-bundles with canonical metrics.
\newblock {\em J. Differential Geom.}, 28(3):361--382, 1988.

\bibitem{DL}
Song Dai and Qiongling Li.
\newblock Minimal surfaces for {H}itchin representations.
\newblock {\em J. Differential Geom.}, 112(1):47--77, 2019.

\bibitem{Dai2018}
Song Dai and Qiongling Li.
\newblock {On cyclic Higgs bundles}.
\newblock {\em Mathematische Annalen}, 376(3–4):1225--1260, November 2020.

\bibitem{D}
S.~K. Donaldson.
\newblock Twisted harmonic maps and the self-duality equations.
\newblock {\em Proc. London Math. Soc. (3)}, 55(1):127--131, 1987.

\bibitem{Ev}
Parker Evans.
\newblock Geometric structures for the $\textrm{G}2'$-hitchin component.
\newblock {\em Advances in Mathematics}, 462:110091, 2025.

\bibitem{Fan}
Yue Fan.
\newblock Construction of the moduli space of {H}iggs bundles using analytic methods.
\newblock {\em Math. Res. Lett.}, 29(4):1011--1048, 2022.

\bibitem{G-P}
Oscar Garc\'{\i}a-Prada.
\newblock Higgs bundles and surface group representations.
\newblock In {\em Moduli spaces and vector bundles}, volume 359 of {\em London Math. Soc. Lecture Note Ser.}, pages 265--310. Cambridge Univ. Press, Cambridge, 2009.

\bibitem{GP}
Oscar Garc\'ia-Prada.
\newblock Vinberg pairs and {H}iggs bundles.
\newblock In {\em Moduli spaces and vector bundles---new trends}, volume 803 of {\em Contemp. Math.}, pages 199--222. Amer. Math. Soc., Providence, RI, 2024.

\bibitem{GPRi}
Oscar Garc\'ia-Prada and S.~Ramanan.
\newblock Involutions and higher order automorphisms of {H}iggs bundle moduli spaces.
\newblock {\em Proc. Lond. Math. Soc. (3)}, 119(3):681--732, 2019.

\bibitem{GPG}
Oscar García-Prada and Miguel González.
\newblock Cyclic higgs bundles and the toledo invariant, 2024.

\bibitem{Gilbarg2001}
David Gilbarg and Neil~S. Trudinger.
\newblock {\em Elliptic Partial Differential Equations of Second Order}.
\newblock Springer Berlin Heidelberg, 2001.

\bibitem{Hi}
N.~J. Hitchin.
\newblock Lie groups and {T}eichm\"{u}ller space.
\newblock {\em Topology}, 31(3):449--473, 1992.

\bibitem{Hitchin:1986vp}
Nigel~J. Hitchin.
\newblock {The Selfduality equations on a Riemann surface}.
\newblock {\em Proc. Lond. Math. Soc.}, 55:59--131, 1987.

\bibitem{Humph}
James~E. Humphreys.
\newblock {\em Reflection groups and {C}oxeter groups}, volume~29 of {\em Cambridge Studies in Advanced Mathematics}.
\newblock Cambridge University Press, Cambridge, 1990.

\bibitem{humphreys1972introduction}
J.E. Humphreys.
\newblock {\em Introduction to Lie Algebras and Representation Theory}.
\newblock Graduate texts in mathematics. Springer, 1972.

\bibitem{L2}
Fran\c{c}ois Labourie.
\newblock Cyclic surfaces and {H}itchin components in rank 2.
\newblock {\em Ann. of Math. (2)}, 185(1):1--58, 2017.

\bibitem{Li}
Qiongling Li.
\newblock An introduction to {H}iggs bundles via harmonic maps.
\newblock {\em SIGMA Symmetry Integrability Geom. Methods Appl.}, 15:Paper No. 035, 30, 2019.

\bibitem{LM1}
Qiongling Li and Takuro Mochizuki.
\newblock Complete solutions of toda equations and cyclic higgs bundles over non-compact surfaces.
\newblock {\em International Mathematics Research Notices}, 2025(7):rnaf081, 04 2025.

\bibitem{Mcintosh}
Ian McIntosh.
\newblock The geometric toda equations for noncompact symmetric spaces.
\newblock {\em Differential Geometry and its Applications}, 99:102249, 2025.

\bibitem{Nieconvex}
Xin Nie.
\newblock Entropy degeneration of convex projective surfaces.
\newblock {\em Conform. Geom. Dyn.}, 19:318--322, 2015.

\bibitem{Nie}
Xin Nie.
\newblock Cyclic {H}iggs bundles and minimal surfaces in pseudo-hyperbolic spaces.
\newblock {\em Adv. Math.}, 436:Paper No. 109402, 75, 2024.

\bibitem{OLIVE1983470}
D.~Olive and Neil Turok.
\newblock The symmetries of dynkin diagrams and the reduction of toda field equations.
\newblock {\em Nuclear Physics B}, 215(4):470--494, 1983.

\bibitem{OV}
Arkadij~L. Onishchik and Ernest~Borisovich Vinberg.
\newblock Lie groups and lie algebras iii.
\newblock 1993.

\bibitem{OT}
Charles Ouyang and Andrea Tamburelli.
\newblock Length spectrum compactification of the so0(2,3)-hitchin component.
\newblock {\em Advances in Mathematics}, 420:108997, 2023.

\bibitem{SS}
Nathaniel Sagman and Peter Smillie.
\newblock Unstable minimal surfaces in symmetric spaces of non-compact type, 2022.

\bibitem{Schoen1997}
Shing-Tung~Yau Schoen, Richard~M.
\newblock {\em Lectures on harmonic maps}.
\newblock Number v. 2 in Conference proceedings and lecture notes in geometry and topology. International Press, Cambridge, MA, 1997.

\bibitem{Katz}
Carlos Simpson.
\newblock Katz's middle convolution algorithm.
\newblock {\em Pure Appl. Math. Q.}, 5(2):781--852, 2009.

\bibitem{S2}
Carlos~T. Simpson.
\newblock Constructing variations of {H}odge structure using {Y}ang-{M}ills theory and applications to uniformization.
\newblock {\em J. Amer. Math. Soc.}, 1(4):867--918, 1988.

\bibitem{Simpson_properness}
Carlos~T. Simpson.
\newblock Moduli of representations of the fundamental group of a smooth projective variety {II}.
\newblock {\em Publications Math\'ematiques de l'IH\'ES}, 80:5--79, 1994.

\bibitem{Vin}
\`E.\~B. Vinberg.
\newblock The {W}eyl group of a graded {L}ie algebra.
\newblock {\em Izv. Akad. Nauk SSSR Ser. Mat.}, 40(3):488--526, 709, 1976.

\bibitem{Wie}
Anna Wienhard.
\newblock An invitation to higher {T}eichm\"{u}ller theory.
\newblock In {\em Proceedings of the {I}nternational {C}ongress of {M}athematicians---{R}io de {J}aneiro 2018. {V}ol. {II}. {I}nvited lectures}, pages 1013--1039. World Sci. Publ., Hackensack, NJ, 2018.

\bibitem{Wo}
C.M. Wood.
\newblock Harmonic sections of homogeneous fibre bundles.
\newblock {\em Differential Geometry and its Applications}, 19(2):193--210, 2003.

\end{thebibliography}

\end{document}